\documentclass[12pt, a4]{amsart}
\usepackage{geometry}
\usepackage{amsxtra,amssymb,amsthm,amsmath,amscd,mathrsfs, epsfig, eufrak}
\usepackage{amscd, amsmath, mathrsfs, amssymb, amsthm, amsxtra, bbding, epsfig, eucal, eufrak, graphicx, latexsym, mathrsfs, mathbbol, bbold}
\usepackage[all]{xy}

\usepackage{tikz-cd}

\usepackage{tikz}
\usepackage{appendix}
\usepackage[normalem]{ulem}
\usepackage{soul}
\usepackage{color}

\setstcolor{red}

\setstcolor{blue}

\setstcolor{cyan}

\setstcolor{orange}
\makeatletter
\@namedef{subjclassname@2010}{%
 \textup{2010} Mathematics Subject Classification}
\makeatother

\def \P {{\mathbb P}}
\def \Q {{\mathbb Q}}
\def \R {{\mathbb R}}

\def \Z {{\mathbb Z}}

\def\tr{\mathop{\rm tr}\nolimits}

\def \d {\,{\rm d}}
\def\re{{\Re e\,}}

\def\leq{\leqslant}
\def\geq{\geqslant}
\def\le{\leqslant}
\def\ge{\geqslant}

\newcommand{\f}{\frac}

\theoremstyle{plain}
\newtheorem{theorem}{Theorem}[section]
\newtheorem{proposition}{Proposition}[section]
\newtheorem{lemma}[proposition]{Lemma}
\newtheorem{corollary}[theorem]{Corollary}

\theoremstyle{remark}

\numberwithin{equation}{section}
\addtocounter{footnote}{1}
\DeclareMathOperator{\diag}{diag}

\begin{document}

\title[\tiny Manin's conjecture for singular cubic hypersurfaces]
{Manin's conjecture \\ for singular cubic hypersurfaces}
\author{\tiny Jianya Liu, Tingting Wen \& Jie Wu} 

\address{%
Jianya Liu
\\
School of Mathematics \& Data Science Institute 
\\
Shandong University
\\
Jinan
\\
Shandong 250100
\\
China} 
\email{jyliu@sdu.edu.cn}

\address{%
Tingting Wen
\\
School of Mathematics
\\
Shandong University
\\
Jinan
\\
Shandong 250100
\\
China} 
\email{ttwen@mail.sdu.edu.cn}

\address{%
Jie Wu
\\
CNRS
\\
LAMA 8050
\\
Universit\'e Paris-Est Cr\'eteil
\\
94010 Cr\'eteil cedex
\\
France}
\email{jie.wu@math.cnrs.fr}

\date{\today}

\begin{abstract}
Let $\mathcal{S}_Q$ denote $x^3= Q(y_1, \ldots, y_m)z$
where $Q$ is a primitive positive definite quadratic form in $m$ variables with integer coefficients. 
This $\mathcal{S}_Q$ ranges over a class of singular cubic hypersurfaces as $Q$ varies.  
For $\mathcal{S}_Q$ we prove 
\begin{itemize}
\item[(i)]
Manin's conjecture is true if $Q$ is locally determined with $2\mid m$ and $m\geq 4$; 

\item[(ii)] 
in general Manin's conjecture is true up to a leading constant if $2\mid m$ and $m\geq 6$. 
\end{itemize}
\end{abstract}


\maketitle   	

\section{Introduction} 
\subsection{Manin's conjecture}  
Counting rational points on algebraic varieties is an important problem in arithmetic geometry
and has received lots of attention.
The far-reaching conjecture of Manin \cite{BatMan90} 
has been a driving force in this area, where ideas from analysis, algebra, as well as 
geometry deeply intertwined. The original conjecture of Manin was formulated for smooth Fano varieties. 
It predicts an asymptotic formula for the density of rational points 
up to some height on the variety in question, where in the main term the number of log-powers  
depends on the rank of its Picard group. 
This has been generalized to a large class of singular Fano varieties by 
Batyrev and Tschinkel  in \cite{BatTsc98b}. 

The aim of the paper is to study Manin's conjecture for 
\begin{equation}\label{def:SmQ}
\mathcal{S}_Q: \quad x^3=Q(y_1, \ldots, y_m) z 
\end{equation}
where $Q(y)=Q(y_1, \ldots, y_m)$ is a primitive positive definite quadratic form in $m\geq 3$ variables with integer coefficients. Here $Q$ is {\it primitive} means that there is no common divisor among
the coefficients of all terms. 
Geometrically speaking, $\mathcal{S}_Q$ ranges over a class of cubic hypersurfaces as $Q$ varies. 
This paper can be viewed as a continuation of our previous research \cite{LWZ19},  where Manin's conjecture is 
proved for $\mathcal{S}_Q$ with $Q$  being the diagonal form    
\begin{equation}\label{def:Sm}
Q(y) = y_1^2+\cdots+y_m^2 
\end{equation}
and $m$ is a multiple of $4$. 

It is well-known that for any $\mathcal{S}_Q$ with $m\geq 3$,   
the heuristic of the circle method does not apply, since there are too many 
solutions with $x=z=0$. One 
therefore counts such solutions of \eqref{def:SmQ} that neither $x$ nor 
$z$ vanishes. If a point in $\P^{m+1}$  is represented by $(x, y_1, \dots, y_m, z) \in \Z^{m+2} $ with
coprime coordinates, then 
\begin{equation}\label{def:Height}
H(x: y_1: \ldots : y_m : z) = \max\{|x|, \textstyle\sqrt{Q(y_1, \ldots, y_m)}, |z| \}^{m-1} 
\end{equation} 
is a natural anticanonical height function for rational points on $\mathcal{S}_Q$.  
Let $B$ be a big parameter, and $N_Q(B)$  denote the number of rational points on
$\mathcal{S}_Q$ satisfying 
\begin{equation}\label{def:Box}
H(x: y_1: \ldots : y_m : z)\leq B, \;
x\not=0, \;
z\not=0.
\end{equation}
We need two more parameters to state Manin's conjecture. The first is $\gamma$, 
the number of crepant divisors over $\Q$ of $\phi:  \widetilde{\mathcal{S}}_Q \to \mathcal{S}_Q$ that is a resolution of singularities. 
Note that $\gamma$ depends on $\mathcal{S}_Q$, but not on $\phi$. 
The second parameter is $r :=\mathrm{rank}_{\Q}  \big(\text{Pic}(\mathcal{S}_Q)\big)$ 
that is the Picard rank of $\mathcal{S}_Q$ over $\Q$. 
Note that  $\mathcal{S}_Q$ is singular for any  $m\geq 1$.  For $m \geq 3$, it is easy to see that it has one isolated singular point $[0: 0: \dots :0 :1]$ and 
a continuous  singular locus $\{ x=z=Q(y_1, \dots, y_m)=0 \}$. Resolving the latter singular locus produces two crepant divisors while none for the first
point. Thus the $\gamma=2$ for $\mathcal{S}_Q$ with $m\geq 3$.  For the details of calculations, see \cite{LWZ19}.  
In addition, it is the same as in \cite[Proposition~2.1]{LWZ19} that we have $r=1$. We remark that these assertions are proved in \cite{LWZ19} just for diagonal 
$Q$, but they easily carry over to the present general case of $Q$. Thus Manin's conjecture now takes the following form.

\medskip 

\noindent
{\bf Conjecture M.} (Manin's conjecture for $\mathcal{S}_Q$).  
{\it Let $\mathcal{S}_Q$ be as in \eqref{def:SmQ} with $Q$ being a positive definite quadratic form in $m\geq 3$ variables with integer coefficients. Let $H$ 
and $N_Q(B)$ be as above. Then, as $B\to\infty$,   
\begin{equation}\label{ManinC}
N_Q(B)\sim C_Q B(\log B)^2 
\end{equation}
where $C_Q$ is a positive constant depending on $Q$. }

\vskip 2mm

For a general statement, see also Yasuda \cite[Conjecture 5.6]{Ya}.

\medskip 

The purpose of this paper is to investigate the above conjecture for a wide class of $\mathcal{S}_Q$. Indeed we will 
establish Manin's conjecture for those $\mathcal{S}_Q$ where the quadratic forms $Q$ are locally determined, which is a new concept introduced in this paper and will be explained in \S1.2. 
For those forms $Q$ that are not locally determined, we will prove that Manin's conjecture 
is true at least up to a leading constant, namely 
\begin{equation}\label{ManinC}
N_Q(B)\asymp B(\log B)^2.  
\end{equation} 
This means that the order of magnitude predicted by Manin is correct. 

There has been a rich literature in the direction of research in this paper, see for example    
\cite{FraManTsc89, BatTsc98a, BatTsc98b, Fou98, Bre98, HBM99, BSD07, LWZ19, Zhai2021} 
and the references therein. Thus we will not give a detailed survey here. 
We also remark that, since our $\mathcal{S}_Q$ is singular,  
the powerful results of Davenport \cite{Dav} as well as Heath-Brown \cite{HeaBro} give, 
instead of an asymptotic formula for the density of rational points, just the existence of a non-trivial point on $\mathcal{S}_Q$,   
while requiring at least $m\geq 12$. 

\medskip 

Our estimates for $N_Q(B)$ will be derived from that for $N_Q^*(B)$, the number of integral solutions of 
\eqref{def:SmQ} satisfying 
\begin{equation}\label{def:Box*}
H^*(x, y_1, \ldots , y_m , z)\leq B, \;
x\not=0, \;
z\not=0.
\end{equation}
Here $H^*$ is another height function defined by 
\begin{equation}\label{def:Affineheight}
H^*(x, y_1, \ldots , y_m, z) = \max\{|x|, \textstyle\sqrt{Q(y_1, \ldots, y_m)}, |z| \}
\end{equation} 
for any point $(x, y_1, \ldots , y_m, z)\in {\Z}^{m+2}$. 
Of course the two height functions $H$ and $H^*$ are closely related, 
and $H^*$ is commonly used in the circle method. 

\subsection{Locally determined quadratic forms} 
We fix some notations before going further. 
Let $Q(y)=Q(y_1, \ldots, y_m)$ be a positive definite quadratic form in $m\geq 3$ variables with integer coefficients. 
In Siegel's notation, we write 
\begin{equation}\label{SieForm}
Q(y)
= {}^{t}y M y = \frac{1}{2}A[y]
=  \frac{1}{2}\sum_{1\le i\le m} a_{ii} y_i^2+ \sum_{1\le i<j\le m} a_{ij} y_iy_j,
\end{equation}
where 
$a_{ij}\in\Z$, 
$a_{ii}\in 2\Z$,  
$A=(a_{ij})$ is a symmetric positive definite matrix of rank $m$, and $M = \frac{1}{2}A$ is called the matrix of $Q$. 
The {\it discriminant} $D$ of $Q$ is defined as $D=(-1)^\frac{m}{2}|A|$ if $m$ 
is even, $\frac{1}{2}(-1)^\frac{m+1}{2}|A|$ if $m$ is odd. Write $\|Q\|:=\max |a_{ij}|$. 
These will be frequently used throughout the paper. 

Our analysis in the following sections will depend on the local behavior of solutions to the equation 
\begin{equation}\label{Q(y)=n}
Q(y)=n.
\end{equation} 
Let $p$ be a prime, and let $\delta_p(n, Q)$ be the local density of \eqref{Q(y)=n} at $p$, i.e. 
\begin{equation}\label{def:localdensity}
\delta_p(n, Q)
:= \lim_{\nu\to\infty} p^{-\nu (m-1)} \big|\big\{y\in(\Z/p^\nu\Z)^m \,:\, Q(y)\equiv n\,({\rm mod}\,p^{\nu})\big\}\big|.
\end{equation}
If $p\nmid 2D$ then $\delta_p(n, Q)$ can be written nicely and explicitly that will be given in Lemma~\ref{Iwaniec};  
it turns out that these primes are harmless. 
The bad primes are the divisors of $2D$, and the main difficulty and complexity in the paper 
are caused by these bad primes. 

Let $p\mid2D$, that is $p$ is a bad prime. For any integer $n$, we can write  
\begin{equation}\label{def:nupnnp}
n=p^{\nu_p(n)}n_p
\quad\text{with}\quad
p\nmid n_p, 
\end{equation}
where $\nu_p(n)$ is the $p$-adic order of $n$. 
A primitive positive definite quadratic form $Q$ is \textit{locally determined}  
if for any integer $n$ and any bad prime $p\mid2D$,  we have
\begin{equation}\label{loc/det}
\delta_p(n, Q) = \delta_p(p^{\nu_p(n)}, Q).
\end{equation}
That is the value of $\delta_p(n, Q)$ which is independent of $n_p$ in the notation \eqref{def:nupnnp}. 
We will see in Lemma~\ref{Iwaniec} that \eqref{loc/det} is naturally satisfied when $p\nmid 2D$. 

In \S3.2, we shall give a sufficient condition 
for locally determined quadratic forms, and also investigate some typical examples. 
\subsection{The results} 

The main results of this paper are as follows.

\begin{theorem}\label{thm1}
Let $m=2k\geq 4$ be an even integer, and let $Q$ be a primitive positive definite quadratic form
in $m$ variables with integer coefficients. If $Q$ is locally determined, then Manin's conjecture is true for $\mathcal{S}_Q$. 
More precisely, as $B\to\infty$,   
\begin{equation}\label{NQB}
\begin{cases}
N^*_Q(B)
 = \mathcal{C}^*_Q B^{m-1} P^*_Q(\log B)
+ O\big(\|Q\|^{\frac{m}{4}}B^{m-1-\frac{1}{6}+\varepsilon}\big),
\\\noalign{\vskip 1,5mm}
N_Q(B)
 = \mathcal{C}_Q B P_Q(\log B) 
+ O\big(\|Q\|^{\frac{m}{4}}B^{1-\frac{1}{6(m-1)}+\varepsilon}\big),
\end{cases}
\end{equation} 
where 
\begin{equation}
\label{Constants:calC*QCQ}
\mathcal{C}^*_Q 
:= \frac{2(2\pi)^\frac{m}{2}}{\Gamma(\frac{m}{2})\sqrt{|A|}}, 
\qquad
\mathcal{C}_Q := \frac{\mathcal{C}^*_Q}{(m-1)^2\zeta(m-1)}, 
\end{equation}
$P^*_Q(t)$ is the quadratic polynomial given by \eqref{Cor:SQxx2} in Proposition \ref{Asymp:SQxySWxy}
whose leading coefficient $\mathscr{C}_Q$ is defined as in \eqref{def:scrCQ},
and $P_Q(t)$ is determined by the formula \eqref{def:PQ(t)}.
The implied constants above depend on $m$ and $\varepsilon$ only. 
\end{theorem}

By Proposition \ref{pro:LocallyDetermined} below, the asymptotic formulae \eqref{NQB} hold
for all positive definite quadratic forms satisfying  
\eqref{Cond:LocallyDetermined-odd}--\eqref{Cond:LocallyDetermined-2},
in particular, for all examples discussed in \S3.3.
Here we only state two special cases: 
\begin{itemize}
\item[(a)]
$Q=y_1^2+\cdots+y_m^2$ with $4\mid m$; 
\vskip 1mm
\item[(b)]
quadratic forms $Q$ of level one whose meaning will be explained immediately.  
\end{itemize}
Define the level $\nu (Q)$ of 
$Q$ to be the smallest positive integer $\nu$ such that $\nu A^{-1}$ has integer entries and even diagonal entries. 
A quadratic form $Q$ is of {\it level one} if $\nu (Q)=1$. 
It is known that, for such a form $Q$, we must have $8\mid m$, $|A|=1$, and $A$ is equivalent to $A^{-1}$. 
The existence of of such forms was proved by Minkowski, 
and an example for $m=8$ can be found in Iwaniec \cite[p.~176]{Iwa97}.

The following corollary gives a significant improvement on our previous result \cite[Theorem~7.1]{LWZ19}, 
where only a logarithmic factor is saved in the error terms. For more information about the 
polynomials $P^*_m(t)$ and $P_m(t)$ below, see also \cite[Theorem~7.1]{LWZ19}. 

\begin{corollary}\label{cor:sum-squares}
If $4\mid m$ and $Q=y_1^2+\cdots+y_m^2$, then Manin's conjecture is true for $\mathcal{S}_Q$. 
More precisely, writing $N^*_m(B)$ and $N_m(B)$ in place of $N^*_Q(B)$ and $N_Q(B)$ 
in this special case, we have,  as $B\to\infty$,   
\begin{equation}\label{cor:N*mB-NmB}
\begin{cases}
\displaystyle
N^*_m(B)
= B^{m-1} P^*_m(\log B) + O\big(B^{m-1-\frac{1}{6}+\varepsilon}\big),
\\\noalign{\vskip 1,5mm}
\displaystyle
N_m(B)
= B P_m(\log B) + O\big(B^{1-\frac{1}{6(m-1)}+\varepsilon}\big), 
\end{cases}
\end{equation}
where 
$P^*_m(t)$ and $P_m(t)$ are two related quadratic polynomials with the leading coefficients
$\mathcal{C}^*_m$ given by \eqref{cor:sum-squares:constant} and $\mathcal{C}^*_m/(m-1)^2\zeta(m-1)$, respectively.
The implied constants depends on $m$ and $\varepsilon$ only.
\end{corollary}

When $Q$ is a quadratic forms of level one, 
we write $N^*_E(B)$ and $N_E(B)$ in place of $N^*_Q(B)$ and $N_Q(B)$.

\vskip 1mm

\begin{corollary}\label{cor:level-one}
Manin's conjecture is true for $\mathcal{S}_Q$ with $Q$ being any positive definite quadratic form 
of level one in $m$ variables and integer coefficients.  
More precisely, as $B\to\infty$, 
\begin{equation}\label{cor:N*EB-NEB}
\begin{cases}
\displaystyle
N^*_E(B)
= B^{m-1} P^*_E(\log B) + O\big(B^{m-1-\frac{1}{6}+\varepsilon}\big),
\\\noalign{\vskip 1,5mm}
\displaystyle
N_E(B)
= B P_E(\log B) + O\big(B^{1-\frac{1}{6(m-1)}+\varepsilon}\big),
\end{cases}
\end{equation}
where $P^*_E(t)$ and $P_E(t)$ are quadratic polynomials with the leading coefficients  
$$
\mathcal{C}^*_E
:= \frac{(2\pi)^{\frac{m}{2}} \zeta(\frac{3m}{2}-2)}{(3m-4) \Gamma(\frac{m}{2}) \zeta(\frac{m}{2})}
\prod_{p} \bigg(1-\frac{1}{p}\bigg)^2 
\bigg(1+\frac{2}{p}+\frac{3}{p^\frac{m}{2}}+\frac{2}{p^{m-1}}+\frac{1}{p^m}\bigg),
$$
and $\mathcal{C}_E := \mathcal{C}^*_E/(m-1)^2\zeta(m-1)$. 
The implied constants depend on $m$ and $\varepsilon$ only. 
\end{corollary} 

Of course the error terms in the above two corollaries are not best possible, 
and we will not get into any further improvement here. 
Very recently Zhai \cite{Zhai2021} improved the results of \cite{BDLWZ2019, LWZ19} by obtaining better error terms, 
and also established the above Corollary \ref{cor:level-one} with a better exponent $\frac{1}{4}$ in place of $\frac{1}{6}$.

In order to treat the general $\mathcal{S}_Q$ where $Q$ is not locally determined, 
we need to introduce the following technical assumption on the determinant $|A|$ of $A$ that
\begin{equation}\label{assumption}
\nu_p(|A|)\leq m-4 \; \text{ at all odd primes and} \; \nu_2(|A|)\leq m+1,
\end{equation}

We have the following general result.

\begin{theorem}\label{thm2}
Let $m=2k\geq 6$ be an even integer, and let $Q$ be a primitive positive definite quadratic form
in $m$ variables with integer coefficients, such that 
\eqref{assumption} holds. Then we have, as $B\to\infty$, 
\begin{equation}\label{estimte:N^*QB}
\begin{cases}
\varpi^- \mathcal{C}'^*_Q W^* \leq N^*_Q(B)\leq \varpi^+ \mathcal{C}'^*_Q W^*, 
\\\noalign{\vskip 1,5mm}
\hskip 1,8mm
\varpi^- \mathcal{C}'_Q W \leq N_Q(B)\leq \varpi^+ \mathcal{C}'_Q W,  
\end{cases}
\end{equation}
where $\varpi^{\pm}$ are given by \eqref{def/pi+-} and
$$
\mathcal{C}'^*_Q 
:= \frac{2(2\pi)^\frac{m}{2}}{\Gamma(\frac{m}{2})\sqrt{|A|}L(\frac{m}{2},\chi)},
\qquad
\mathcal{C}'_Q 
:= \frac{\mathcal{C}'^*_Q}{(m-1)^2\zeta(m-1)}\cdot
$$  
Here $\chi(q)=\chi_{4D}(q)=\big(\frac{4D}{q}\big)$ is the Jacobi symbol, 
and $L(s,\chi)$ is the Dirichlet $L$-function attached to $\chi$. 
In addition, 
\begin{equation}\label{def:W*=}
\begin{cases}
W^* := B^{m-1} P^*_W(\log B) +O_{m,\varepsilon}(\|Q\|^{\frac{m}{4}}B^{m-1-\frac{1}{6}+\varepsilon}), 
\\\noalign{\vskip 1mm}
\hskip 1,6mm
W := B P_W(\log B) +O_{m,\varepsilon}(\|Q\|^{\frac{m}{4}}B^{1-\frac{1}{6(m-1)}+\varepsilon}), 
\end{cases}
\end{equation} 
where $P^*_W$ is the quadratic polynomial given by \eqref{Cor:SWxx2} in Proposition \ref{Asymp:SQxySWxy} 
with leading coefficient $\mathscr{C}_W$ defined as in \eqref{def:scrCW} below,
and $P_W(t)$ is defined similar to \eqref{def:PQ(t)}. 
\end{theorem}

We remark that the case of odd $m$ can be treated in the same way, but the technical details are different. 
We leave this to another occasion, in order to minimize the size of the present paper.

\section{The equation $Q(y)=n$}

\subsection{The number of solutions to $Q(y)=n$} 
Proposition~\ref{Prop/Q=n} below is essentially due to Iwaniec \cite[Theorem~11.2]{Iwa97}. 
A difference is that here the implied $O$-constant does not depend on $Q$.   
By applying \cite[Theorem 20.9]{IwaniecKowalski}, we make explicit 
the dependence on $Q$ in the error term.

\begin{proposition}\label{Prop/Q=n}
Let $Q$ be a primitive positive definite quadratic form in $m=2k\geq 4$ variables with integer coefficients, 
and write $Q$ in the form of \eqref{SieForm}.  
Then the number $r(n, Q)$ of integral solutions to $Q(y)=n$ satisfies 
\begin{equation}\label{formula:r(n,Q)}
r(n, Q)
= \frac{(2\pi)^{\frac{m}{2}}n^{\frac{m}{2}-1}}{\Gamma(\frac{m}{2})\sqrt{|A|}}\mathfrak{S}(n,Q)
+ O\big(m^{\frac{3m}{4}}\|Q\|^{\frac{m}{4}}n^{\frac{m-1}{4}+\varepsilon}\big),
\end{equation}
where $\mathfrak{S}(n, Q)$ is the singular series associated to the equation $Q(y)=n$ (see \eqref{def:S(n,Q)} below),  
and the implied constant depends on $\varepsilon$ only. 
\end{proposition}

\begin{proof}
We mainly follow the proof of \cite[Theorem~11.2]{Iwa97} and keep its notations for convenience as far as possible. 
Let $\widetilde{Q}$ be the adjoint quadratic form defined by $\widetilde{Q}(u) := \f12 A^{-1}[u]$.
We have
\begin{equation}\label{rnQ}
r(n, Q) = 2\,\re\bigg(\sum_{c\leq C}\int_{0}^{1/(cC)}T(c,n;x) {\rm e}(-nz)\d x\bigg),
\end{equation}
where $z=x+{\rm i}y$ with any $y>0$ to be chosen later, 
\begin{equation}\label{T}
T(c, n; x)
:= \frac{({\rm i}/z)^\frac{m}{2}}{\sqrt{|A|}c^m} \sum_{u\in\Z^m}T_u(c, n; x), 
{\rm e}\Big(-\frac{\widetilde{Q}(u)}{c^2z}\Big), 
\end{equation}
and
$$
T_u(c, n; x)
:= \sum_{\substack{C<d\leq c+C\\cdx<1, \, (c,d)=1}} {\rm e}\Big(n\frac{\bar{d}}{c}\Big) G_u\Big(-\frac{d}{c}\Big)
$$
with the Gauss sum $G_u$ defined by 
\begin{equation*}
G_u\Big(\frac{d}{c}\Big) := \sum_{h\,({\rm mod}\;c)} {\rm e}\Big(\frac{d}{c}(Q(h)+{}^th u)\Big).
\end{equation*}
We are going to need the bound 
\begin{equation}\label{GauBd}
\Big|G_u\Big(\frac{d}{c}\Big)\Big|\le (c m \|Q\|)^\frac{m}{2} 
\end{equation}
which will be established in Lemma~\ref{lem3} below. 
We divide the complete sum $T_u(c, n; x)$ into two sums to moduli $c_0, c_1$ with $c_0c_1=c$, $(c_1, 2|A|)=1$ and $c_0$ having all prime factors in $2|A|$. 
Replacing \cite[Lemma 20.12]{IwaniecKowalski} with our \eqref{GauBd}, we have
\begin{equation}\label{lem5}
T_u(c, n; x)
\ll (m\|Q\|)^{\frac{m}{2}}(n+\widetilde{Q}(u),c_1)^{\frac{1}{2}}c_0^{\frac{1}{2}}c^{\frac{m+1}{2}}\tau(c)\log 2c,
\end{equation}
where $\widetilde{Q}(u)=\frac{1}{2}A^{-1}[u]$ as before and the implied constant does not depend on $Q$.

Observing that $T_u(c, n; x)$ is constant for $0<x<\frac{1}{c(c+C)}$, 
we apply the bound \eqref{lem5} to all terms in \eqref{T} except for $m=0$ in the range $0<x<\frac{1}{c(c+C)}$ in which $T_0(c,n;x)$ is equal to
\begin{equation*}\label{Tcn}
T(c, n)
= \sum_{\substack{d\,({\rm mod}\,c)\\ (d, c)=1}} {\rm e}\Big(n\frac{\bar{d}}{c}\Big) 
\sum_{h\,({\rm mod}\;c)} {\rm e}\Big(-\frac{d}{c}Q(h)\Big).
\end{equation*}
Then we obtain that
\begin{align*}
& T(c,n;x)
\\
& = \frac{T(c, n)}{\sqrt{|A|}\,c^m} \Big(\frac{{\rm i}}{z}\Big)^{\frac{m}{2}}\!
+ O\bigg(\frac{(m\|Q\|)^{\frac{m}{2}}(c_0c)^{\frac{1}{2}}\tau(c)\log 2c}{\sqrt{|A|}(c|z|)^{\frac{m}{2}}}
\sideset{}{'}\sum_{u\in\Z^m}(n+\widetilde{Q}(u),c_1)^{\frac{1}{2}}\exp\Big(\!\!-\!\frac{2\pi y\widetilde{Q}(u)}{c^2|z|^2}\Big)\bigg),
\end{align*}
where $\sum'$ means that $m=0$ is excluded from the summation if $0<x<\frac{1}{c(c+C)}$. 
We have
$$
\sideset{}{'}\sum\leq \Big(\sum_{\ell\geq 0} \frac{(n\nu+\ell, c_1)^{\frac{1}{2}}}{(1+\ell)^2}\Big)
\sideset{}{'}\sum_{u\in\Z^m}(1+\nu\widetilde{Q}(u))^2\exp\Big(-\frac{2\pi y\widetilde{Q}(u)}{c^2|z|^2}\Big),
$$
where $\nu=\nu(Q)$ is the level of $Q$ and $(c_1,2\nu)=1$.

Suppose $0<\lambda_1\leq\lambda_2\leq\cdots\leq\lambda_m=:\lambda$ are the eigenvalues of $Q$. 
Then 
$$
\sum_{1\leq i \leq m}\lambda_i
= \tr(A) 
= \sum_{1\leq i\leq m}a_{ii}\leq m\|Q\|.
$$ 
It follows that $\lambda_m\leq m\|Q\|$, and consequently 
$$
\widetilde{Q}(u)\geq \frac{|u|^2}{2\lambda} \ge \frac{|u|^2}{2m\|Q\|}, 
$$
where $|u|$ is the $l^1$-norm of the vector $u$.  
Hence for any $0<x<1/(cC)$, taking $C=n^{\frac{1}{2}}, y=C^{-2}=n^{-1}$, we obtain trivially
$$
\sideset{}{'}\sum_{u\in\Z^m}(1+\nu\widetilde{Q}(u))^2\exp\Big(\!\!-\frac{2\pi y\widetilde{Q}(u)}{c^2|z|^2}\Big) 
\ll (c|z|m^{\frac{1}{2}}\|Q\|^{\frac{1}{2}}y^{-\frac{1}{2}})^\kappa$$
for any $\kappa>0$. Taking $\kappa=\frac{m}{2}$, we get 
\begin{align*}
T(c,n;x)
= \frac{T(c, n)}{\sqrt{|A|}\,c^m}\Big(\frac{{\rm i}}{z}\Big)^{\frac{m}{2}}
+ O\Big(\frac{(m\|Q\|)^{\frac{3m}{4}} \xi(c_1)(c_0c)^{\frac{1}{2}}\tau(c)\log 2c}{\sqrt{|A|}}n^{\frac{m}{4}}\Big)
\end{align*}
where 
$$
\xi(c_1) 
:= \sum_{\ell\geq 0}(n\nu+\ell,c_1)^{\frac{1}{2}}(1+\ell)^{-2}\ll 1.
$$
Inserting this into \eqref{rnQ}, we obtain
\begin{align*}
r(n, Q)
& = \frac{1}{\sqrt{|A|}}\sum_{c\leq C} \frac{T(c, n)}{c^m} 
\int_{-1/(cC)}^{1/(cC)} \Big(\frac{{\rm i}}{z}\Big)^{\frac{m}{2}} {\rm e}(-nz)\d x
\\
& \quad
+ O\bigg(\frac{(m\|Q\|)^{\frac{3m}{4}}}{\sqrt{|A|}}n^{\frac{m}{4}}\sum_{c\leq C}(c_0c)^{\frac{1}{2}}\tau(c)(\log 2c)\int_{-1/(cC)}^{1/(cC)} {\rm e}(-nz)\d x\bigg).
\end{align*}
And the error term is bounded by
$$
\frac{(m\|Q\|)^{\frac{3m}{4}}}{\sqrt{|A|}} n^{\frac{m}{4}}
\sum_{c\leq C} \frac{(c_0c)^{\frac{1}{2}}\tau(c)\log 2c}{cC}
\ll \frac{m^{\frac{3m}{4}}\|Q\|^{\frac{m}{4}}n^{\frac{m}{4}}}{C^{1-\varepsilon}}\sum_{c\leq C} \frac{1}{\sqrt{c_1}}
\ll m^{\frac{3m}{4}}\|Q\|^{\frac{m}{4}}n^{\frac{m-1}{4}+\varepsilon}.
$$
Finally, the error terms from the main term contribute no more than the above. These give the error term in 
Proposition \ref{Prop/Q=n}. 
\end{proof} 

\begin{lemma}\label{lem3} 
Let $Q$ be a primitive positive definite quadratic form in $m=2k\geq 4$ variables with integer coefficients, 
and write $Q$ in the form of \eqref{SieForm}. For $(c, d)=1$ and $u\in \Z^m$, define   
\begin{equation*}
G_u\Big(\frac{d}{c}\Big) := \sum_{h\,({\rm mod}\;c)} {\rm e}\Big(\frac{d}{c}(Q(h)+{}^th u)\Big).
\end{equation*}
Then we have
\begin{equation*}
\Big|G_u\Big(\frac{d}{c}\Big)\Big|\le (c m \|Q\|)^\frac{m}{2}.
\end{equation*}
\end{lemma}	

\begin{proof}
For $(c, d)=1$, we have
\begin{align*}
\Big|G_u\Big(\frac{d}{c}\Big)\Big|^2
& = \sum_{h, \, h'\,({\rm mod}\,c)} {\rm e}\Big(\frac{d}{c}\big(Q(h')-Q(h)+{}^t(h'-h) u\big)\Big)
\\
& = \sum_{h, \, z\,({\rm mod}\,c)} 
{\rm e}\Big(\frac{d}{c}\Big(\frac{1}{2}\,{}^t\!z A h+\frac{1}{2}\,{}^t\!hAz+Q(z)+{}^t\!z u\Big)\Big)
\\\noalign{\vskip 1mm}
& \leq c^m |\{z \,({\rm mod}\,c) : A z\equiv 0\,({\rm mod}\,c)\}|.
\end{align*}
Write $A=(a_{ij})_{m\times m} = {}^t(\alpha_1,\alpha_2,\dots,\alpha_m)$, 
and let $|\alpha|$ denote the $l^1$-norm of a vector $\alpha$.    
There is only one solution for the system $Az=d$ of $m$ linear equations in $m$ variables if $|A|\neq 0$. 
The congruent system $Az \equiv 0\,({\rm mod}\,c)$ can be written as the linear system $Az=cb$ 
with $b={}^t(b_1, \dots, b_m)$, where $b_i$ only takes integer value in the finite range $[1, |\alpha_i|]$. 
That is, there will be at most 
$$
\prod\limits_{i=1}^m|\alpha_i|\leq (m\|Q\|)^m
$$ 
systems, and each system has only one solution since $|A|\neq0$. So the number of solutions 
of the system $Az\equiv 0\,({\rm mod}\,c)$ is at most $ (m\|Q\|)^m$, i.e.	
$$
|\{z \,({\rm mod}\,c) : A z\equiv 0\,({\rm mod}\,c)\}|
\le (m\|Q\|)^m.
$$
Inserting this back, we obtain the desired bound for $|G_u|$. The lemma is proved. 
\end{proof}


\section{Locally determined quadratic forms}
The aim of this section is to give a sufficient condition for a quadratic form to be locally determined. 
This will be done in Proposition \ref{pro:LocallyDetermined}, and after that we will 
offer some examples of locally determined quadratic forms and give a uniform estimate for local densities at bad primes. 

\subsection{Explicit formulae for local densities of quadratic forms} 
We shall quote explicit formulae of local densities $\delta_p(n, Q)$, which will be useful later.
Conserving the notations \eqref{def:localdensity} and \eqref{def:nupnnp}, these formulae can be stated in three different cases:
$$
\left\{
\begin{array}{lll}  
& \!\!\!\!p\nmid 2D, \\ 
& \!\!\!\!\text{odd} \; p\mid D, \\ 
& \!\!\!\!p=2, 
\end{array} 
\right.
$$
where $D$ is the discriminant of the quadratic form $Q$. Primes in the first case are ordinary, but  
those in the second and third cases are bad. 

For ordinary primes we quote Iwaniec \cite[(11.72)]{Iwa97}. 
The formula below shows that we indeed have \eqref{loc/det} in this case. 

\begin{lemma}[Local densities at ordinary primes]\label{Iwaniec} 
Let $Q$ be a primitive positive definite quadratic form in $m=2k\geq 4$ variables with integer coefficients, 
and write $Q$ in the form of \eqref{SieForm}. For $p\nmid 2D$, we have 
$$
\delta_p(n, Q) 
= \bigg(1-\frac{\chi(p)}{p^k}\bigg)
\bigg(1-\frac{\chi(p)}{p^{k-1}}\bigg)^{-1}
\bigg(1-\frac{\chi(p^{\nu_p(n)+1})}{p^{(\nu_p(n)+1)(k-1)}}\bigg), 
$$
where $\chi(q)=\chi_{4D}(q)=\big(\frac{4D}{q}\big)$ is the Jacobi symbol. 
\end{lemma}

Now we turn to the case of bad primes, where the detailed computations of Yang \cite{Yang} 
will be very important. For odd prime $p\mid D$, we may assume that the 
matrix $M$ of $Q$ is $\Z_p$-equivalent to 
$$
\diag(\varepsilon_1 p^{\alpha_1}, \dots, \varepsilon_m p^{\alpha_m})
\;\;\text{with}\;\; 
\varepsilon_h\in\Z_p^*
\;\;\text{for}\;\; 1\leq h\leq m
\;\;\text{and}\;\; 
0\leq\alpha_1\leq \cdots \leq \alpha_m,
$$
which we call the normalization of $Q$ over $\Z_p$. 
The following is Yang \cite[Theorem 3.1]{Yang}. 

\begin{lemma}[Local densities at odd bad primes]\label{Yang3.1} 
Let $Q$ be a primitive positive definite quadratic form in $m=2k\geq 4$ variables with integer coefficients, 
and write $Q$ in the form of \eqref{SieForm}. 
For any odd prime $p\mid D$, we have 
\begin{equation}\label{Yang:odd-prime}
\delta_p(n, Q)
= 1 + \bigg(1-\frac{1}{p}\bigg) \sum_{\substack{1\leq r\leq \nu_p(n)\\ \ell(r)\, \text{\rm even}}} \frac{v_r}{p^{d(r)}}
+ \frac{v_{\nu_p(n)+1}f(n)}{p^{d(\nu_p(n)+1)}},
\end{equation}
where 
\begin{align*}
L(r) 
& := \{1\leq h\leq m : \alpha_h-r<0 \; \text{is odd}\},
\qquad\;\;\,
\ell(r) := |L(r)|,
\\
f(n)
& := \begin{cases}
-\frac{1}{p}  & \text{if $2\mid \ell(\nu_p(n)+1)$},
\\
(\frac{n_p}{p})\frac{1}{\sqrt{p}} &\text{if $2\nmid \ell(\nu_p(n)+1)$},
\end{cases}
\hskip 18,3mm
\delta_p
:= \begin{cases}
1          & \text{if $p\equiv 1 \: ({\rm mod}\,4)$},
\\
\hskip 0,5mm
{\rm i}  & \text{if $p\equiv 3 \: ({\rm mod}\,4)$},
\end{cases}	
\\
d(r)
& := \frac{1}{2}\sum_{\alpha_h<r}(r-\alpha_h)-r,
\hskip 36,8mm
v_r
:= \delta_p^{[\frac{1}{2}\ell(r)]} \prod\limits_{h\in L(r)} \bigg(\frac{\varepsilon_h}{p}\bigg), 
\end{align*}
$\big(\frac{q}{p}\big)$ is the Legendre symbol and $[t]$ is the integral part of $t$.
\end{lemma}

When $p=2$, the matrix $M$ of $Q$ is $\Z_2$-equivalent to
\begin{equation*}\label{Z2equiv}
\diag(\tilde{\varepsilon}_1 2^{\tilde{\alpha}_1}, \dots, \tilde{\varepsilon}_R 2^{\tilde{\alpha}_R})
\oplus\bigg\{\bigoplus_{i=1}^S \varepsilon_i' 2^{\beta_i} 
\begin{pmatrix}
0 & \frac{1}{2}
\\\noalign{\vskip 1mm}
\frac{1}{2} & 0
\end{pmatrix}\bigg\}
\oplus \bigg\{\bigoplus_{j=1}^T \varepsilon_j''2^{\gamma_j} 
\begin{pmatrix}
1 & \frac{1}{2}
\\\noalign{\vskip 1mm}
\frac{1}{2} & 1
\end{pmatrix}\bigg\},
\end{equation*}
where $\tilde{\varepsilon}_h$, $\varepsilon'_i$, $\varepsilon''_j\in\Z_2^*$, 
$\tilde{\alpha}_h\geq 0$, $\beta_i\geq 0$, $\gamma_j\geq 0$ are all integers and $R+2S+2T=m$.
Thus we have the normalization 
$$
Q = \sum_{h=1}^{R} \tilde{\varepsilon}_h 2^{\tilde{\alpha}_h} x_h^2
+ \sum_{i=1}^S\varepsilon_i' 2^{\beta_i}y_{i1}y_{i2}
+ \sum_{j=1}^T\varepsilon_j'' 2^{\gamma_j}(z_{j1}^2+z_{j1}z_{j2}+z_{j2}^2).
$$
The following lemma is Yang \cite[Theorem 4.1]{Yang}.

\begin{lemma}[Local density at $2$]\label{Yang4.1} 
Let $Q$ be a primitive positive definite quadratic form in $m=2k\geq 4$ variables with integer coefficients, 
and write $Q$ in the form of \eqref{SieForm}. 
We have 
\begin{equation}\label{Yang:2-prime}
\begin{aligned}
\delta_2(n, Q)
& = 1 + \sum_{\substack{r=1\\ \tilde{\ell}(r-1) \, {\rm odd}}}^{\nu_2(n)+3}
\bigg(\frac{2}{\mu_r(n)\varepsilon(r)}\bigg) \frac{\delta(r) p(r)}{2^{\tilde{d}(r)+\frac{3}{2}}}
\\
& \hskip 7,2mm
+ \sum_{\substack{r=1\\ \tilde{\ell}(r-1) \, {\rm even}}}^{\nu_2(n)+3}
\bigg(\frac{2}{\varepsilon(r)}\bigg) 
\frac{\delta(r) p(r) \psi(\tfrac{1}{8}\mu_r(n)) {\rm char}(4\Z_2)(\mu_r(n))}{2^{\tilde{d}(r)+1}} ,
\end{aligned}
\end{equation}
where 
\begin{align*}
\tilde{d}(r)
& := \frac{1}{2}\sum_{\tilde{\alpha}_h<r-1}(r-1-\tilde{\alpha}_h) 
+ \sum_{\beta_i<r} (r-\beta_i) + \sum_{\gamma_j<r} (r-\gamma_j)-r,
\\
\tilde{L}(r)
& := \{1\leq h\leq R : \tilde{\alpha}_h-r<0 \; \text{is odd}\},
\hskip 12mm
\tilde{\ell}(r) := |\tilde{L}(r)|,
\\\noalign{\vskip 1mm}
\kappa(r)
& := \sum_{1\le h\le R, \, \tilde{\alpha}_h<r-1} \tilde{\varepsilon}_h,
\hskip 42,5mm
\mu_r(n)
:= 2^{\nu_2(n)+3-r}n_2 - \kappa(r),
\\\noalign{\vskip 1mm}
\Big(\frac{2}{t}\Big)
& := \begin{cases}
(-1)^{\frac{1}{8}(t^2-1)}  & \text{if $t\in\Z_2^*$},
\\
0                          & \text{otherwise},
\end{cases}
\hskip 25mm	
{\rm char}(4\Z_2)(t)
:= \begin{cases}
1 & \text{if $t\equiv 0 \, ({\rm mod}\,4)$},
\\ 
0 & \text{otherwise},
\end{cases}
\\\noalign{\vskip 1mm}
\delta(r)
& := \begin{cases}
0  & \text{if $\tilde{\alpha}_h=r-1$ for some $h$},
\\
1  & \text{otherwise},
\end{cases}
\hskip 15mm
p(r) := (-1)^{\sum_{\gamma_j<r}(\gamma_j-r)},
\end{align*}
$\psi(t) := {\rm e}^{-2\pi{\rm i}t}$ and  
$\varepsilon(r) := \prod_{h\in \tilde{L}(r-1)} \tilde{\varepsilon}_h$.
\end{lemma}

\subsection{A sufficient condition for locally determined quadratic forms} 
The following proposition gives a sufficient condition for $Q$ to be locally determined.

\begin{proposition}\label{pro:LocallyDetermined}
{\rm (i)}
A primitive positive definite quadratic form $Q$ in $m$ variables is locally determined 
provided
\begin{equation}\label{Cond:LocallyDetermined-odd}
\ell(r)\equiv 0 \, ({\rm mod}\,2)
\end{equation}
for all integers $r\ge 1$ and odd primes $p\mid D$; 
 and
\begin{equation}\label{Cond:LocallyDetermined-2}
\tilde{\ell}(r)\equiv 0 \, ({\rm mod}\,2),
\qquad
\kappa(r)\equiv 0 \, ({\rm mod}\,4)
\end{equation}
for all integers $r\ge 1$. Then 
we have
\begin{equation}\label{odd-prime:final}
\delta_p(n, Q)
= 1 + \bigg(1-\frac{1}{p}\bigg) \sum_{r=1}^{\nu_p(n)} \frac{v_r}{p^{d(r)}}
+ \frac{v_{\nu_p(n)+1}}{p^{d(\nu_p(n)+1)+1}}
= \delta_p(p^{\nu_p(n)}, Q)
\end{equation}
for odd prime $p\mid D$; and
\begin{equation}\label{2-prime:final}
\delta_2(n, Q)
= 1 + \sum_{r=1}^{\nu_2(n)+1}
\bigg(\frac{2}{\varepsilon(r)}\bigg) 
\frac{(-1)^{\frac{1}{4}\kappa(r)+[r/(\nu_2(n)+1)]} \delta(r) p(r) }{2^{\tilde{d}(r)+1}}
= \delta_2(2^{\nu_2(n)}, Q).
\end{equation}

\par
{\rm (ii)}
Under the conditions \eqref{Cond:LocallyDetermined-odd}--\eqref{Cond:LocallyDetermined-2}, 
$|A|$ is a perfect square.
\end{proposition}

\begin{proof} 
Recall that a primitive positive definite quadratic form $Q$ is called locally determined if 
$\delta_p(n, Q) = \delta_p(p^{\nu_p(n)}, Q)$ for all integers $n$ and all primes $p\mid 2D$.

Since $\ell(r)$ is always even and $f(n)=-1/p$ under this assumption, 
then the formula \eqref{odd-prime:final} is an immediate consequence of 
\eqref{Yang:odd-prime} of Lemma \ref{Yang3.1}.
Noticing that the right-hand side of \eqref{odd-prime:final} is independent of $n_p$, 
we have $\delta_p(n, Q) = \delta_p(p^{\nu_p(n)}, Q)$ for all integers $n\ge 1$ and odd primes $p\mid D$.

\vskip 0,5mm

Next we consider the case of $p=2$.
Since $\tilde{\ell}(r)$ is always even, 
the first sum on the right-hand side of \eqref{Yang:2-prime} of Lemma \ref{Yang4.1} is empty. 
Thus
\begin{equation}\label{2-prime:1}
\delta_2(n, Q)
= 1 + \sum_{r=1}^{\nu_2(n)+3}
 \bigg(\frac{2}{\varepsilon(r)}\bigg)\frac{\delta(r) p(r)  \psi(\tfrac{1}{8}\mu_r(n)) {\rm char}(4\Z_2)(\mu_r(n))}{2^{\tilde{d}(r)+1}}.
\end{equation}
In view of the assumption that $\kappa(r)\equiv 0\,({\rm mod}\,4)$ for all $r\ge 1$,
the definitions of $\mu_r(n)$ and of ${\rm char}(4\Z_2)(\mu_r(n))$ imply that
$$
{\rm char}(4\Z_2)(\mu_r(n)) 
= \begin{cases}
1 & \text{for $1\le r\le \nu_2(n)+1$},
\\\noalign{\vskip 0,5mm}
0 & \text{for $\nu_2(n)+2\le r\le \nu_2(n)+3$}.
\end{cases}
$$
On the other hand, we have
$$
\psi(\tfrac{1}{8}\mu_r(n))
= {\rm e}^{\pi {\rm i}(\frac{1}{4}\kappa(r)-2^{\nu_2(n)+1-r}n_2)}
= (-1)^{\frac{1}{4}\kappa(r)+[r/(\nu_2(n)+1)]} 
$$
for $1\le r\le \nu_2(n)+1$. Combining these with \eqref{2-prime:1}, 
we find \eqref{2-prime:final}.
Since the right-hand side of \eqref{2-prime:final} is independent of $n_2$, 
we have $\delta_2(n, Q) = \delta_2(2^{\nu_2(n)}, Q)$ for all integers $n\ge 1$.

\vskip 1mm

Then we prove {\rm (ii)}.
We have known that $Q$ is $\Z_p$-equivalent to diagonal forms at odd prime $p$ and 
$\sum_{h=1}^m\alpha_h=\nu_p(|A|)$. 
The condition \eqref{Cond:LocallyDetermined-odd} means that the number of terms under the same power is even, which implies the $p$-adic order of $|A|$ is even. 

Next we consider the case $p=2$.
According to \cite[p.119]{Cassels}, any positive definite quadratic form, over $\Z_2$, can be transformed into the linear combination of diagonal terms and binary positive definite forms 
$$
h(y_{i1},y_{i2})=h_{11}y_{i1}^2+2h_{12}y_{i1}y_{i2}+h_{22}y_{i2}^2
$$ 
with odd $h_{12}$ and even $h_{11}$, $h_{22}$, which are $\Z_2$-equivalent to 
$$
2y_{i1}y_{i2}\quad\text{or}\quad 2y_{i1}^2+2y_{i1}y_{i2}+2y_{i2}^2
$$
according as $$h_{11}h_{22}-h_{12}^2\equiv1\,(\rm mod\,8)\quad\text{or}\quad \equiv5\,({\rm mod}\,8).$$
The first condition $\tilde{\ell}(r)\equiv0\,({\rm mod}\,2)$ of \eqref{Cond:LocallyDetermined-2} implies that $\sum_{h=1}^R\tilde{\alpha}_h$ is even. And the power of 2 coming from non-diagonal terms is always even. 
Thus $|A|$ is a perfect square. 
\end{proof}

\subsection{Examples} 
Now we give some examples of locally determined quadratic forms.

\medskip 

{\sc Example 1}.  
If $m=4k$ then $Q(y) = y_1^2 + \cdots + y_m^2$ is locally determined. 

Note that $D=2^m$ is the discriminant of $Q$.
There is no odd prime $p\mid D$.
We only need to consider the bad prime $2$. 
It is clear that $Q$ is $\Z_2$-equivalent to itself. Comparing with the normalization over $\Z_2$, we have $R=m$, $S=T=0$ and $(\tilde{\alpha}_h, \tilde{\varepsilon}_h) = (0, 1)$ for $1\le h\le m$.
These imply that
$$
\tilde{L}(r) = \begin{cases}
\emptyset       & \text{if $2\mid r$}, 
\\
\{1, \dots, m\} & \text{if $2\nmid r$},
\end{cases}
\qquad
\tilde{\ell}(r) = \begin{cases}
0       & \text{if $2\mid r$},
\\
m & \text{if $2\nmid r$},
\end{cases}
\qquad
\kappa(r) = \begin{cases}
0 & \text{if $r=1$},
\\
m & \text{if $r\ge 2$}.
\end{cases}
$$
Thus the condition \eqref{Cond:LocallyDetermined-2} is satisfied for all integers $r\ge 1$, since $m=4k$. 
This proves that $Q$ is locally determined. Note that
$$
\delta(1)=0, \;\delta(r)=1 \;\;\text{for}\;\;r\geq2, \quad
p(r)=\varepsilon(r)=1
\quad\text{and}\quad
\tilde{d}(r)=2k(r-1)-r.
$$ 
By \eqref{2-prime:final}, we have
\begin{equation}\label{sum-square:delta2}
\delta_2(n, Q)
= 1 + \sum_{r=2}^{\nu_2(n)} \frac{(-1)^{k}}{2^{(2k-1)(r-1)}} - \frac{(-1)^{k}}{2^{(2k-1)\nu_2(n)}}\cdot
\end{equation}

\medskip

{\sc Example 2}.  
Let $Q(y)=\frac{1}{2}A[y]$ be a positive definite quadratic form in $m$ variables with integer coefficients. 
If $Q$ is of level one, then $Q$ is locally determined. 

To prove this, we first note that $Q$ being of level one implies $8\mid m$ and $|A|=1$. 
Hence $D=1$, and it suffices to consider the bad prime $2$. The matrix $M=\tfrac{1}{2}A$ of $Q$ is $\Z_2$-equivalent to 
$$
\bigoplus\limits_{i=1}^S
\begin{pmatrix}
0 & \frac{1}{2}
\\
\frac{1}{2} & 0
\end{pmatrix}, 
$$ 
that is $R=T=0$ and $(\beta_i, \varepsilon'_i) = (0, 1)$ for $1\leq i\leq S=\frac{m}{2}$.
This implies $\tilde{L}(r)=\emptyset$ and $\tilde{\ell}(r)=\kappa(r) = 0$ for all integers $r\ge 1$.
Thus the condition \eqref{Cond:LocallyDetermined-2} is satisfied for all integers $r\ge 1$. 
This proves that $Q$ is locally determined. Further we have
$$
\delta(r)=p(r)=\varepsilon(r)=1,
\qquad
\tilde{d}(r)=(\tfrac{m}{2}-1)r.
$$
By \eqref{2-prime:final}, we have
\begin{equation}\label{level-one:delta2}
\begin{aligned}
\delta_2(n, Q)
& = 1+\sum_{r=1}^{\nu_2(n)}2^{(1-\frac{m}{2})r-1}-2^{(1-\frac{m}{2})(\nu_2(n)+1)-1}
\\
& = \frac{1-2^{-\frac{m}{2}}}{1-2^{1-\frac{m}{2}}} \big(1-2^{(1-\frac{m}{2})(\nu_2(n)+1)}\big).
\end{aligned}
\end{equation} 

\medskip 

{\sc Example 3}.  
The quadratic form
$$
Q(y) = y_1^2+3y_2^2+y_3^2+y_3y_4+y_4^2.
$$
is locally determined. 

To prove this, we note that the matrix $M$ of $Q$ is equal to
$$
M = \begin{pmatrix}
1 & 0 & 0 & 0 
\\
0 & 3 & 0 & 0 
\\
0 & 0 & 1 & \frac{1}{2}
\\
0 & 0 & \frac{1}{2} & 1
\end{pmatrix}
= \diag(1, 3)
\oplus
\begin{pmatrix}
1 & \frac{1}{2}
\\
\frac{1}{2} & 1
\end{pmatrix}.
$$
Since the discriminant $D = (-1)^{\frac{m}{2}}|A| = |2M| = 2^23^2$, 
we only need to consider the local densities at $p=2$ and $p=3$. 
Noticing that the matrix $M$ of $Q$ is $\Z_2$-equivalent to itself, we have $R=2$, $S=0$, $T=1$ and
$$
\tilde{\alpha}_1=\tilde{\alpha}_2=0,
\quad
\tilde{\varepsilon}_1=1, 
\quad
\tilde{\varepsilon}_2=3,
\quad
\gamma_1=0,
\quad
\varepsilon''_1=1.
$$
From these, we easily see that
$$
\tilde{\ell}(r) = \begin{cases}
0 & \text{if $2\mid r$},
\\
2 & \text{if $2\nmid r$},
\end{cases}
\qquad\text{and}\qquad
\kappa(r) = \begin{cases}
0 & \text{if $r=1$},
\\
4 & \text{if $r\ge 2$}.
\end{cases}
$$
Thus the condition \eqref{Cond:LocallyDetermined-2} is verified for all integers $r\ge 1$.

Now let
$$
U = \begin{pmatrix}
1
\\
{} & 1
\\
{} & {} & -1 & 1
\\
{} & {} &  1 & 1
\end{pmatrix}.
$$ 
This is a $3$-adic unit matrix and $|U|=-2$. Noticing that
$$
{}^tUMU=\diag\{1, 3, 1, 3\},
$$
we have
$$
\alpha_1=\alpha_3=0
\quad
\alpha_2=\alpha_4=1,
\quad
\varepsilon_1=\varepsilon_2=\varepsilon_3=\varepsilon_4=1.
$$
These imply that
$$
\ell(r) = \begin{cases}
2 & \text{if $r$ is odd}, 
\\
2 & \text{if $r$ is even}. 
\end{cases}
$$
Thus the condition \eqref{Cond:LocallyDetermined-odd} is verified for the prime 3 and all integers $r\ge 1$. 
Consequently $Q$ is locally determined.

\subsection{Lower and upper bounds for $\varpi(n, Q)$} 
We now multiply the local densities at all bad primes together to define 
\begin{equation}\label{def:vpi}
\varpi(n, Q) := \prod\limits_{p| 2D} \delta_p(n, Q).
\end{equation}
The function $\varpi(n, Q)$ is rather complicated, since each of 
these $\delta_p(n, Q)$ is not only involved but also not multiplicative in $n$. 
We can bound $\varpi(n, Q)$ from below and above in the following form. 

\begin{proposition}\label{SinSerEv}
Let $Q$ be a primitive positive definite quadratic form in $m=2k\geq 6$  variables with integer coefficients, 
and write $Q$ in the form of \eqref{SieForm}. Assume \eqref{assumption}. Then
\begin{equation}\label{pi-pi+}
\varpi^- \leq \varpi(n, Q)\leq \varpi^+
\end{equation}
for all integers $n\ge 1$, where
\begin{equation}\label{def/pi+-}
\varpi^{\pm} 
:= \Big(1\pm\frac{49}{50}\Big) \prod_{p\mid D, \, p\geq 3} \Big(1\pm \frac{1}{p}\Big). 
\end{equation}
\end{proposition}
\begin{proof}
Firstly we consider the case of odd prime factor $p$ of $D$. According to \eqref{Yang:odd-prime}, we have
\begin{equation*}
\big|\delta_p(n,Q)-1\big|
\leq \left(1-\frac{1}{p}\right)\sum_{1\leq r\leq \nu_p(n)}\frac{1}{p^{d(r)}}+\frac{f(n)}{p^{d(\nu_p(n)+1)}}
\leq \left(1-\frac{1}{p}\right)\sum_{r\geq 1}\frac{1}{p^{d(r)}}\cdot
\end{equation*}
It suffices to determine what $d(r)$ looks like. If $r> \alpha_m$, we know that every $\alpha_h$ is counted in the summation of $d(r)$, then $$d(r)=\frac{1}{2}\sum_{1\leq h\leq m}(r-\alpha_h)-r=\frac{m-2}{2}r-\frac{\nu_p(|A|)}{2},$$
since $\sum_{h=1}^{m}\alpha_h=\nu_p(|A|)$. 
And if $r\leq \alpha_m$, there are some terms $\alpha_h$ not satisfying $\alpha_h<r$ that should not be counted in $d(r)$, then we have $d(r)\geq\frac{m-2}{2}r-\frac{\nu_p(|A|)}{2}$.
It follows that 
\begin{equation}\label{Bound:oddprime}
\big|\delta_p(n,Q)-1\big|
\leq(1-\tfrac{1}{p})\sum_{r\geq 1}\frac{1}{p^{\frac{m-2}{2}r-\frac{\nu_p(|A|)}{2}}}
=(1-\tfrac{1}{p})\frac{p^{-\frac{m-2}{2}+\frac{\nu_p(|A|)}{2}}}{1-p^{-\frac{m-2}{2}}}
\leq\frac{p^{\frac{\nu_p(|A|)}{2}}}{p^{\frac{m-2}{2}}}\leq\frac{1}{p},
\end{equation}
since $\nu_p(|A|)\leq m-4$.

Next we treat the case of prime factor 2 of $2D$. According to \eqref{Yang:2-prime}, we have
\begin{align*}\label{den/2/c}
|\delta_2(n, Q)-1|\leq \sum_{r=1}^{\nu_2(n)+3} \frac{\delta(r)}{2^{\tilde{d}(r)+1}}
\end{align*}
if ignoring the parity of $\tilde{\ell}(r-1)$. 
And the assumption $\nu_2(|A|)\le m+1$ gives
\begin{equation}\label{sum2}
\sum_{1\leq h\leq R}(\tilde{\alpha}_h+1)
+ 2\Big(\sum_{1\leq i\leq S} \beta_i+\sum_{1\leq j\leq T} \gamma_j\Big)\leq m+1.
\end{equation}
It can be observed as before that if $r>\max\{\tilde{\alpha}_h+1,\beta_i,\gamma_j\}$, we have
\begin{align*}
\tilde{d}(r)
& = \frac{1}{2}\sum_{1\leq h\leq R}(r-1-\tilde{\alpha}_h)+\sum_{1\leq i\leq S}(r-\beta_i)+\sum_{1\leq j\leq T}(r-\gamma_j)-r
\\
& = \Big(\frac{R}{2}+S+T\Big)r-r-\frac{1}{2}\Big(\sum_{1\leq h\leq R} (\tilde{\alpha}_h+1)
+ 2\sum_{1\leq i\leq S} \beta_i + 2\sum_{1\leq j\leq T} \gamma_j\Big)
\\
&=\frac{m-2}{2}r-\frac{\nu_2(|A|)}{2}\cdot
\end{align*}
And $\tilde{d}(r)\geq \frac{m-2}{2}r-\frac{\nu_2(|A|)}{2}$ 
holds if $r$ is smaller.

Note that $Q$ is primitive, which means that there is at least one $\tilde{\alpha}_h$ or $\beta_i$ or $\gamma_j$ equal to $0$. Otherwise, there will be a common divisor among
the coefficients of all terms. 
If $\tilde{\alpha}_h=0$ for some $h$, then $\delta(1)=0$ according to the definition. It follows that
\begin{align*}
\sum_{r=1}^{\nu_2(n)+3} \frac{\delta(r)}{2^{\tilde{d}(r)+1}}
\leq\sum_{r\geq2}\frac{1}{2^{\tilde{d}(r)+1}} =\f{2^{-m+1+\frac{\nu_2(|A|)}{2}}}{1-2^{-\f{m-2}{2}}}
=\frac{2^\frac{\nu_2(|A|)}{2}}{2^{m-1}-2^\frac{m}{2}}
\leq \frac{\sqrt{2}}{3}
\end{align*}
since $m\geq 6$ and $\nu_2(|A|)\le m+1$. 

If there exist one $\beta_i$ or $\gamma_j$ equal to $0$, this one should be counted in the summation of $\tilde{d}(1)$. So we have $\tilde{d}(1)\geq 0$ and
\begin{align*}
|\delta_2(n, Q)-1|
\leq \frac{1}{2^{\tilde{d}(1)+1}}+\sum_{r\geq2} \frac{\delta(r)}{2^{\tilde{d}(r)+1}}
\leq\frac{1}{2}+\sum_{r\geq2} \frac{1}{2^{\tilde{d}(r)+1}}
\leq \frac{1}{2}+\frac{\sqrt{2}}{3}\cdot
\end{align*} 
It follows that 
\begin{equation}\label{Bound:2prime}
\big|\delta_2(n, Q)-1\big|\leq \frac{49}{50}\cdot
\end{equation}
Now the required inequalities follow immediately from \eqref{Bound:oddprime} and \eqref{Bound:2prime}.
\end{proof}

\section{Outline of the proof and Dirichlet series}\label{Dirichlet}
\subsection{Outline of the proof} 

The estimate for $N_Q(B)$ will be deduced from that for 
$N^*_Q(B)$ by the M\"obius inversion in the following manner
\begin{equation}\label{MobInv}
N_Q(B)
= \sum_{d\leq B^{1/(m-1)}} \mu(d) N_Q^*(B^{1/(m-1)}/d), 
\end{equation}
where $\mu(d)$ is the M\"obius function. To investigate $N^*_Q(B)$, the key observation is that 
\begin{equation}\label{def:NdaggerQ(B)}
N^*_Q(B)
= 2 \sum_{1\leq a \leq B} \sum_{1\leq n \leq B^2} \mathbb{1}_3(an) r(n, Q),
\end{equation}
where $r(n, Q)$ is the number of integral solutions to the equation $Q(y)=n$, and 
$\mathbb{1}_3$ is the characteristic function
\begin{equation}\label{def:13n}
\mathbb{1}_3(n)
:= \begin{cases}
1 & \text{if $n$ is a cube}, 
\\
0 & \text{otherwise}. 
\end{cases}
\end{equation}
The next step is to insert the asymptotic formula for $r(n, Q)$ in Proposition~\ref{Prop/Q=n}
into the above \eqref{def:NdaggerQ(B)}, so that we have to handle 
\begin{equation}\label{def:SQxy}
S_Q(x, y):= \sum_{a\leq x} \sum_{n\leq y} \mathbb{1}_3(an) n^{k-1} \mathfrak{S}(n, Q),
\end{equation}
where $\mathfrak{S}(n, Q)$ is the singular series associated to the equation $Q(y)=n$. We are going to 
estimate $S_Q(x, y)$ by methods from multiplicative number theory, and to this end we have to know 
the arithmetic properties of $\mathfrak{S}(n, Q)$ in detail. 

The singular series is defined by 
\begin{equation}\label{def:S(n,Q)}
\mathfrak{S}(n, Q) 
:= \sum_{c=1}^\infty \frac{1}{c^m}\sum_{\substack{d\,({\rm mod}\,c)\\ (c, d)=1}}
\sum_{h_1 ({\rm mod}\,c)} \cdots \sum_{h_m ({\rm mod}\,c)}
{\rm e}\bigg(\frac{d}{c}\big(Q(h_1, \dots, h_m)-n\big)\bigg). 
\end{equation}
It can be written as the product of all local densities 
\begin{equation}\label{SinSer=:2}
\mathfrak{S}(n, Q)
= \prod_p \delta_p(n, Q), 
\end{equation}
which will be handled directly if $Q$ is locally determined. Otherwise, by Lemma~\ref{Iwaniec}, it is further written as
\begin{equation}\label{SinSer=:2}
\mathfrak{S}(n, Q)
= \prod_p \delta_p(n, Q)
= \frac{\sigma_{1-k}(n, \chi)}{L(k, \chi)} \varpi(n, Q), 
\end{equation}
where $L(k, \chi)$ is the Dirichlet $L$-function attached to $\chi$, 
\begin{equation}\label{def:sigmaellchi-pi}
\sigma_{1-k}(n, \chi) := \sum_{d| n} \chi(d) d^{1-k}, 
\end{equation}
and $\varpi(n, Q)$ is as in \eqref{def:vpi}. This explains why we spend such efforts to 
understand $\varpi(n, Q)$.  
In order to prove Theorems~\ref{thm1} and \ref{thm2}, 
we are going to study $S_Q(x, y)$ and 
\begin{align}\label{def:SWxy}
S_W(x, y) := \sum_{a\leq x} \sum_{n\leq y} \mathbb{1}_3(an) n^{k-1} \sigma_{1-k}(n, \chi)
\end{align}
respectively.
Our results are as follows.

\begin{proposition}\label{Asymp:SQxySWxy}
Let $Q$ be a locally determined quadratic form in $m=2k\geq 4$ variables with integer coefficients.  
Then for any $\varepsilon >0$ we have
\begin{align}
S_Q(x, y)
& = x^{\frac{1}{3}} y^{k-\frac{2}{3}} \big\{P_Q(\log x, \log y)
+ O_{k, \varepsilon}\big(x^{-\frac{1}{6}+\varepsilon} + x^{\frac{1}{12}+\varepsilon} y^{-\frac{1}{6}}\big)\big\}
\label{Evaluation:SQxy}, 
\\
S_W(x, y)
& = x^{\frac{1}{3}} y^{k-\frac{2}{3}} \big\{P_W(\log x, \log y)
+ O_{k, \varepsilon}\big(x^{-\frac{1}{6}+\varepsilon} + x^{\frac{1}{12}+\varepsilon} y^{-\frac{1}{6}}\big)\big\}, 
\label{Evaluation:Skxy}
\end{align}
uniformly for $x\ge 2$ and $y\geq 2$, 
where $P_Q(t, u)$ and $P_W(t, u)$ are quadratic polynomials, 
and the implied constant depends on $k$ and $\varepsilon$ only. 
In particular, we have
\begin{align}
S_Q(x, x^2)
& = x^{m-1} P_Q(\log x)+O_{m, \varepsilon}(x^{m-\frac{7}{6}+\varepsilon}),
\label{Cor:SQxx2}
\\
S_W(x, x^2)
& = x^{m-1} P_W(\log x)+O_{m, \varepsilon}(x^{m-\frac{7}{6}+\varepsilon}),
\label{Cor:SWxx2}
\end{align}
where $P_Q(t)$ and $P_W(t)$ are quadratic polynomials with leading coefficients $\mathscr{C}_Q$ and $\mathscr{C}_W$
given by 
\begin{equation}\label{def:scrCQ}
\begin{aligned}
\displaystyle
\mathscr{C}_Q
& := \frac{L(\frac{3}{2}m-2,\chi^3)}{(6m-8)L(\frac{m}{2}, \chi)}
\prod_{p\,\nmid\,2D} \bigg(1-\frac{1}{p}\bigg)^2
\bigg(1+\frac{2}{p}+\frac{3\chi}{p^\frac{m}{2}}+\frac{2\chi^2}{p^{m-1}}+\frac{\chi^2}{p^m}\bigg)
\\
& \hskip 33,5mm\times
\prod_{p| 2D} 
\bigg\{\bigg(1-\frac{1}{p}\bigg)^4\sum_{d\geq 0} \frac{\sum_{\nu=0}^{3d} \delta_p(p^{\nu}, Q)}{p^d}\bigg\}
\end{aligned}
\end{equation}
and
\begin{equation}\label{def:scrCW}
\mathscr{C}_W
:= \frac{L(\frac{3}{2}m-2,\chi^3)}{6m-8}
\prod_{p} \bigg(1-\frac{1}{p}\bigg)^2
\bigg(1+\frac{2}{p}+\frac{3\chi}{p^\frac{m}{2}}+\frac{2\chi^2}{p^{m-1}}+\frac{\chi^2}{p^m}\bigg),
\end{equation}
respectively.
\end{proposition}

The remaining part of the paper is devoted to the proof of Proposition~\ref{Asymp:SQxySWxy}. 

\subsection{Dirichlet series}  
We begin to study the following Dirichlet series associated with $S_Q(x, y)$ and $S_W(x, y)$: 
\begin{align}
\mathfrak{F}(s, w) 
& := \sum_{a\ge 1} \sum_{n\ge 1} \frac{\mathbb{1}_3(an) n^{k-1}\mathfrak{S}(n, Q)}{a^s n^w},
\label{def:scrFsw}
\\
\mathcal{F}(s, w) 
& := \sum_{a\ge 1} \sum_{n\ge 1} \frac{\mathbb{1}_3(an) n^{k-1}\sigma_{1-k}(n, \chi)}{a^s n^w},
\label{def:Fsw}
\end{align}
for $\re s>\frac{1}{3}$ and $\re w>k-\frac{2}{3}$. 

\begin{lemma}\label{Lem:FQsw}
Let $Q$ be a locally determined quadratic form in $m=2k\geq 4$ variables. 
For  $\re s>\frac{1}{3}$ and $\re w>k-\frac{2}{3}$,
we have
\begin{align}
\mathfrak{F}(s, w) 
& = \prod_{0\le j\le 3} \zeta\big((3-j)s+j(w-k+1)\big) \mathfrak{G}(s, w),
\label{Expression:frakFsw}
\\
\mathcal{F}(s, w) 
& = \prod_{0\le j\le 3} \zeta\big((3-j)s+j(w-k+1)\big) \mathcal{G}(s, w),
\label{Expression:calFsw}
\end{align}
where $\mathfrak{G}(s, w)$ and $\mathcal{G}(s, w)$ are Euler products 
given by \eqref{def:scrGsw} and \eqref{def:Gsw} below.
Further, for any $\varepsilon>0$, $\mathfrak{G}(s, w)$ and $\mathcal{G}(s, w)$ converge absolutely for 
\begin{equation}\label{cond:sigmau}
\min_{0\le j\le 3} \re \big((3-j)s+j(w-k+1)\big)\ge \tfrac{1}{2}+\varepsilon,
\end{equation} and
in this half-plane
\begin{equation}\label{UB:Gsw}
\mathfrak{G}(s, w)\ll_{k, \varepsilon} 1,
\qquad
\mathcal{G}(s, w)\ll_{k, \varepsilon} 1.
\end{equation}
\end{lemma}

\begin{proof}
Since $Q$ is locally determined, we have
$\delta_p(n, Q) = \delta_p(p^{\nu_p(n)}, Q)$ for all primes $p\mid 2D$,
where $\nu_p(n)$ is the $p$-adic order of $n$.
This also holds for all primes $p\nmid 2D$ by Lemma~\ref{Iwaniec}. 
Hence we can write the formal Euler product
\begin{align*}
\mathbb{1}_3(an) n^{k-1}\mathfrak{S}(n, Q)
& = \prod_p \mathbb{1}_3(p^{\nu_p(a)+\nu_p(n)}) p^{(k-1)\nu_p(n)} \delta_p(n, Q)
\\
& = \prod_p \mathbb{1}_3(p^{\nu_p(a)+\nu_p(n)}) p^{(k-1)\nu_p(n)} \delta_p(p^{\nu_p(n)}, Q)
\end{align*}
for all integers $a\ge 1$ and $n\ge 1$. 
For convenience we write $\mu=\nu_p(a)$, $\nu=\nu_p(n)$.
This formula allows us to express the Dirichlet series $\mathfrak{F}(s, w)$ in form of the Euler product
\begin{equation}\label{exp:scrFsw}
\begin{aligned}
\mathfrak{F}(s, w)
& = \prod_p \bigg(\mathop{\sum_{\mu\ge 0} \sum_{\nu\ge 0}}_{3| (\mu+\nu)} 
\frac{\delta_p(p^{\nu}, Q)}{p^{\mu s+\nu(w-k+1)}}\bigg)
\\
& = \prod_{p\,\nmid\,2D} \mathfrak{F}_p(s, w) \prod_{p| 2D} \mathop{\sum_{\mu\ge 0} \sum_{\nu\ge 0}}_{3| (\mu+\nu)} 
\frac{\delta_p(p^{\nu}, Q)}{p^{\mu s+\nu(w-k+1)}}
\end{aligned}
\end{equation}
for $\re s>\frac{1}{3}$ and $\re w>k-\frac{2}{3}$, where
$$
\mathfrak{F}_p(s, w)
:= \bigg(1-\frac{\chi(p)}{p^k}\bigg)
\bigg(1-\frac{\chi(p)}{p^{k-1}}\bigg)^{-1}
\mathop{\sum_{\mu\ge 0} \sum_{\nu\ge 0}}_{3| (\mu+\nu)} 
\frac{1-\chi(p^{\nu+1})/p^{(\nu+1)(k-1)}}{p^{\mu s+\nu(w-k+1)}}
$$
for $p\nmid 2D$.
Since $n\mapsto \chi(n)$ is completely multiplicative, we have $\chi(p^{\nu+1})=\chi(p)^{\nu+1}=\chi^{\nu+1}$.
Changing variables $\mu+\nu=3d$, we deduce that
\begin{align*}
\mathfrak{F}_p(s, w)
& = \frac{1-\chi p^{-k}}{1-\chi p^{-(k-1)}}\sum_{d\ge 0} \sum_{0\leq \nu \leq 3d} 
\frac{1-(\chi p^{-(k-1)})^{\nu+1}}{p^{(3d-\nu)s+\nu(w-k+1)}}
\\
& = \frac{1-\chi p^{-k}}{1-\chi p^{-(k-1)}}
\sum_{d\ge 0} \frac{1}{p^{3ds}} 
\bigg(\frac{1-p^{-(3d+1)(w-s-k+1)}}{1-p^{-(w-s-k+1)}} 
- \frac{\chi}{p^{k-1}} \frac{1-(\chi p^{-(w-s)})^{3d+1}}{1-\chi p^{-(w-s)}}\bigg).
\end{align*}
Further we have
\begin{align*}
\sum_{d\ge 0} \frac{1}{p^{3ds}}\cdot \frac{1-p^{-(3d+1)(w-s-k+1)}}{1-p^{-(w-s-k+1)}} 
& = \frac{1}{1-p^{-(w-s-k+1)}} \sum_{d\ge 0} \bigg(\frac{1}{p^{3ds}} - \frac{p^{-(w-s-k+1)}}{p^{3d(w-k+1)}}\bigg) 
\\
& = \frac{1+p^{-(2s+w-k+1)}+p^{-(s+2w-2k+2)}}{(1-p^{-3s})(1-p^{-3(w-k+1)})}, 
\end{align*}
and
\begin{align*}
\sum_{d\ge 0} \frac{\chi p^{-(k-1)}}{p^{3ds}}\cdot \frac{1-(\chi p^{-(w-s)})^{3d+1}}{1-\chi p^{-(w-s)}}
& = \frac{\chi p^{-(k-1)}}{1-\chi p^{-(w-s)}} \sum_{d\ge 0} \bigg(\frac{1}{p^{3ds}} - \chi p^{-(w-s)}\frac{\chi^{3d}}{p^{3dw}}\bigg) 
\\
& = \frac{\chi p^{-(k-1)}+\chi^2 p^{-(2s+w+k-1)}+\chi^3 p^{-(s+2w+k-1)}}{(1-p^{-3s})(1-\chi^3 p^{-3w})}\cdot
\end{align*}
Inserting these two formulae back, we have 
\begin{align*}
\mathfrak{F}_p(s, w)
& = \frac{(1-\chi p^{-k})(1-\chi^3 p^{-3w})^{-1}\mathfrak{F}^*_p(s, w)}{(1-\chi p^{-(k-1)})(1-p^{-3s})(1-p^{-3(w-k+1)})},
\end{align*}
where
\begin{align*}
\mathfrak{F}^*_p(s, w)
& = \big(1-\chi^3 p^{-3w}\big) \big(1+p^{-(2s+w-k+1)}+p^{-(s+2w-2k+2)}\big)
\\
& \quad
- \big(1-p^{-3(w-k+1)}\big) \big(\chi p^{-(k-1)}+\chi^2 p^{-(2s+w+k-1)}+\chi^3 p^{-(s+2w+k-1)}\big).
\end{align*}
An elementary computation shows that
\begin{align*}
\mathfrak{F}^*_p(s, w)
& = (1-\chi p^{-(k-1)})\bigg(1+\frac{1}{p^{2s+w-k+1}}
+\frac{1}{p^{s+2w-2k+2}}
\\
& \quad
+\frac{\chi}{p^{2s+w}}
+\frac{\chi}{p^{s+2w-k+1}}
+ \frac{\chi}{p^{3w-2k+2}}
+\frac{\chi^2}{p^{s+2w}}
+ \frac{\chi^2}{p^{3w-k+1}}
+\frac{\chi^2}{p^{2s+4w-2k+2}}\bigg).
\end{align*}
Combining these two formulae, for $p\nmid 2D$, we can write
\begin{equation}\label{def:Fpsw}
\mathfrak{F}_p(s, w)
= \frac{(1-\chi p^{-k})(1-\chi^3p^{-3w})^{-1} \mathfrak{G}_p(s, w)}
{\prod_{j=0}^3 (1-p^{-((3-j)s+j(w-k+1))})}
\end{equation}
with
\begin{equation}\label{def:frakGpsw-pnmid2D}
\begin{aligned}
\mathfrak{G}_p(s, w)
& := \bigg(1-\frac{1}{p^{2s+w-k+1}}\bigg)
\bigg(1-\frac{1}{p^{s+2(w-k+1)}}\bigg)
\bigg(1+\frac{1}{p^{2s+w-k+1}}
+\frac{1}{p^{s+2(w-k+1)}}
\\
& \quad
+\frac{\chi}{p^{2s+w}}
+\frac{\chi}{p^{s+2w-k+1}}
+ \frac{\chi}{p^{3w-2k+2}}
+\frac{\chi^2}{p^{s+2w}}
+ \frac{\chi^2}{p^{3w-k+1}}
+\frac{\chi^2}{p^{2s+4w-2k+2}}
\bigg).
\end{aligned}
\end{equation}
This proves \eqref{Expression:frakFsw} with 
\begin{equation}\label{def:scrGsw}
\mathfrak{G}(s, w)
:= \frac{L(3w, \chi^3)}{L(k, \chi)} \prod_{p} \mathfrak{G}_p(s, w)
\end{equation}
and
\begin{equation}\label{def:frakGpsw-pmid2D}
\mathfrak{G}_p(s, w)
:= \prod_{0\le j\le 3} \bigg(1-\frac{1}{p^{(3-j)s+j(w-k+1)}}\bigg)
\bigg(\sum_{d\ge 0} \sum_{0\le \nu\le 3d} \frac{\delta_p(p^{\nu}, Q)}{p^{(3d-\nu) s+\nu(w-k+1)}}\bigg)
\end{equation}
for $p\mid 2D$.

The formula \eqref{Expression:calFsw} is much easier to prove, since
$n\mapsto \mathbb{1}_3(an) n^{k-1}\sigma_{1-k}(n, \chi)$ is multiplicative.
We can write the formal Euler product
\begin{equation}\label{exp:calFsw}
\mathcal{F}(s, w)
= \prod_p \bigg(\mathop{\sum_{\mu\ge 0} \sum_{\nu\ge 0}}_{3| (\mu+\nu)} 
\frac{\sigma_{1-k}(p^{\nu}, \chi)}{p^{\mu s+\nu(w-k+1)}}\bigg)
=: \prod_p \mathcal{F}_p(s, w)
\end{equation}
for $\re s>\frac{1}{3}$ and $\re w>k-\frac{2}{3}$.
On the other hand, since $n\mapsto \chi_{4D}(n)$ is completely multiplicative, we have
$$
\sigma_{1-k}(p^{\nu}, \chi)
= \sum_{0\le j\le \nu} (\chi(p) p^{-(k-1)})^j
= \frac{1-(\chi(p) p^{-(k-1)})^{\nu+1}}{1-\chi(p) p^{-(k-1)}}\cdot
$$
Similar to \eqref{def:Fpsw} and \eqref{def:frakGpsw-pnmid2D}, we can prove that
$$
\mathcal{F}_p(s, w)
= \frac{(1-\chi^3p^{-3w})^{-1}\mathfrak{G}_p(s, w)}
{\prod_{j=0}^3 (1-p^{-((3-j)s+j(w-k+1))})},
$$
where $\mathfrak{G}_p(s, w)$ is defined as in \eqref{def:frakGpsw-pnmid2D} for all primes $p$.
This proves the formula \eqref{Expression:calFsw} with
\begin{equation}\label{def:Gsw}
\mathcal{G}(s, w) := L(3w, \chi^3) \prod_{p} \mathfrak{G}_p(s, w).
\end{equation}

Next we prove \eqref{UB:Gsw}. 
Note that \eqref{def:frakGpsw-pmid2D} is a finite product, which is bounded for $p\mid 2D$.
Set $s=\sigma+\mathrm{i}\tau$ and $w=u+\mathrm{i}v$.
It is easy to verify that under the condition \eqref{cond:sigmau}, 
we have
\begin{align*}
2(2\sigma+u-k+1)
& \ge 2(\tfrac{1}{2}+\varepsilon) = 1+2\varepsilon,
\\
2(\sigma+2(u-k+1))
& \ge 2(\tfrac{1}{2}+\varepsilon) = 1+2\varepsilon,
\\
2\sigma+u
& \ge k-1+\tfrac{1}{2}+\varepsilon = k-\tfrac{1}{2}+\varepsilon,
\\
\sigma+2u-k+1
& \ge k-1+\tfrac{1}{2}+\varepsilon = k-\tfrac{1}{2}+\varepsilon,
\\
3u-2k+2
& \ge k-1+\tfrac{1}{2}+\varepsilon = k-\tfrac{1}{2}+\varepsilon,
\\
\sigma+2u
& \ge 2k-2+\tfrac{1}{2}+\varepsilon = 2k-\tfrac{3}{2}+\varepsilon,
\\
3u-k+1
& \ge 2k-2+\tfrac{1}{2}+\varepsilon = 2k-\tfrac{3}{2}+\varepsilon,
\\
2\sigma+4u-2k+2
& \ge 2k-2+\tfrac{1}{2}+\varepsilon = 2k-\tfrac{3}{2}+\varepsilon.
\end{align*}
These together with \eqref{def:frakGpsw-pnmid2D} imply that for $p\nmid 2D$,
$$
\mathfrak{G}_p(s, w)
= 1 + O_{k, \varepsilon}(p^{-1-\varepsilon})
$$
provided \eqref{cond:sigmau} holds,
which implies that the Euler products $\mathfrak{G}(s, w)$ and $\mathcal{G}(s, w)$ converge absolutely
and \eqref{UB:Gsw} holds in this half-plane.
By analytic continuation, \eqref{Expression:frakFsw} and \eqref{Expression:calFsw} are also true in the same domain.
This completes the proof.
\end{proof}

\vskip 8mm

\section{Difference operator}

The basic idea is to apply the method of complex integration to,
instead of our original $S_Q(x, y)$ and $S_W(x, y)$, the quantity
\begin{align}
M_Q(X, Y)
& := \int_1^Y \int_1^X S_Q(x, y) \d x \d y,
\label{def:MQXY}
\\
M_W(X, Y)
& := \int_1^Y \int_1^X S_W(x, y) \d x \d y,
\label{def:MkXY}
\end{align}
which are mean values of $S_Q(x, y)$ and $S_W(x, y)$. 
We remark that $M_Q(X, Y)$ and $M_W(X, Y)$ are much easier to handle.
We will first establish an asymptotic formula for $M_Q(X, Y)$ (resp. $M_W(X, Y)$), and then
derive the asymptotic formula \eqref{Evaluation:Skxy} for $S_Q(x, y)$ (resp. $S_W(x, y)$) 
by an analytic argument involving the difference operator $\mathscr{D}$ defined by
\begin{equation}\label{def:Df}
(\mathscr{D}f)(X, H ;\: Y, J) := f(H, J) - f(H, Y) - f(X, J) + f(X, Y).
\end{equation}
The two quantities $S_W(x, y)$ and $M_W(X, Y)$ are closed related as shown in the following lemma,
which in particular enables one to derive an asymptotic formula for $S_Q(x, y)$ (resp. $S_W(x, y)$)
from that for $M_Q(X, Y)$ (resp. $M_W(X, Y)$).

\begin{lemma}\label{lem3.1}
Let $S_Q(X, Y)$ and $M_Q(X, Y)$ be defined as in \eqref{def:SQxy} and \eqref{def:MQXY}.
Then
$$
(\mathscr{D}M_Q)(X-H, X ;\: Y-J, Y)\le HJS_Q(X, Y)\le (\mathscr{D}M_Q)(X, X+H; \: Y, Y+J)
$$
for $1\le H\le \frac{1}{2}X$ and $1\le J\le \frac{1}{2}Y$.
The same inequalities also hold if we replace $S_Q(X, Y)$ and $M_Q(X, Y)$ by
$S_W(x, y)$ and $M_W(X, Y)$, respectively.
\end{lemma}

The difference operator $\mathscr{D}$ has some properties that we are going to use repeatedly
throughout the paper. These are summarized in the following lemma.

\begin{lemma}\label{lem3.2}
\par
{\rm (i)}
Let $ f(x,y)$ be a function of class $C^3$. Then we have
$$
(\mathscr{D}f)(X, H; \: Y, J)
= (J-Y)(H-X) \bigg\{\frac{\partial^2f}{\partial x\partial y}(X, Y) + O\big(R(X, H; \:Y, J)\big)\bigg\}
$$
for $X\le H$ and $Y\le J$, where
$$
R(X, H; \: Y, J)
:= (H-X)
\max_{\substack{X\le x\le H\\ Y\le y\le J}} \bigg|\frac{\partial^3f}{\partial x^2\partial y}(x, y)\bigg|
+ (J-Y)\max_{\substack{X\le x\le H\\ Y\le y\le J}} \bigg|\frac{\partial^3f}{\partial x\partial y^2}(x, y)\bigg|.
$$
\par
{\rm (ii)}
If $f(X, Y) = f_1(X)f_2(Y)$, then
$$
(\mathscr{D}f)(X, H; \: Y, J) = \big(f_1(H)-f_1(X)\big) \big(f_2(J)-f_2(Y)\big).
$$
\end{lemma}

Lemmas~\ref{lem3.1} and \ref{lem3.2} can be proved similarly as in
\cite[Lemma 2]{Bre98} and the details are therefore omitted.
The next elementary estimate is essential and it will also be used several times in the next scetion \cite[Lemma 6(i)]{Bre98}.

\begin{lemma}\label{Lem3.3}
Let $1\le H\le X$ and $|\sigma|\le 10$.
Then for any $\beta\in [0, 1]$, we have
\begin{equation}\label{Lem3.2_Eq_A}
\big|(X+H)^{s} - X^{s}\big|\ll X^{\sigma} ((|\tau|+1)H/X)^{\beta},
\end{equation}
where the implied constant is absolute.
\end{lemma}

\section{Proof of Proposition \ref{Asymp:SQxySWxy}}\label{ProofPropositionSkxy}

We shall evaluate $M_Q(X, Y)$ defined by \eqref{def:MkXY}, from which we shall deduce
\eqref{Evaluation:SQxy} of Proposition \ref{Asymp:SQxySWxy} by applying the difference operator $\mathscr{D}$ 
defined as in \eqref{def:Df}. 
The asymptotic formula \eqref{Evaluation:Skxy} can be proved in the exactly same way.
We suppose
\begin{equation}\label{Condition:XYHJ}
1\le H\le \tfrac{1}{2}X,
\qquad
1\le J\le \tfrac{1}{2}Y,
\end{equation}
and fix the following notations: 
\begin{equation}\label{def:kappa_Lambda_L}
\begin{aligned}
s
& := \sigma+\mathrm{i}\tau,
\hskip 14mm
z :=w-k+1= u+\mathrm{i}v,
\hskip 11mm
\mathcal{L}:=\log X,
\\
\kappa
& := \tfrac{1}{3}+\mathcal{L}^{-1},
\qquad
\lambda_k
:= k-\tfrac{2}{3}+2\mathcal{L}^{-1},
\hskip 21,4mm
\lambda := \tfrac{1}{3}+2\mathcal{L}^{-1}.
\end{aligned}
\end{equation}
The following proposition is an immediate consequence of
Lemmas \ref{Perron_Formula:MXY}--\ref{Lem:Evaluate_I3} below.

\begin{proposition}\label{Pro:MQXY}
We have
\begin{align*}
M_Q(X, Y)
& = Y^{k-1} \big\{(XY)^{\frac{4}{3}} \big(P_1(\log Y) + P_2(\log(X^{-\frac{1}{2}}Y)) + P_3(\log(X^{-2}Y)\big)
\\
& \hskip 14mm
+ R_0(X, Y) + R_1(X, Y) + R_2(X, Y) + R_3(X, Y)\big\}
\end{align*}
for $X\ge 2$ and $Y\ge 2$,
where the $R_j(X,Y)$ and $P_j(t)$ are defined as in
\eqref{def:R0XY},
\eqref{def:R1XY},
\eqref{def:R2XY},
\eqref{def:R3XY}
and \eqref{def:P1t}, \eqref{def:P2t}, \eqref{def:P3t} below,
respectively.
\end{proposition}

The proof is divided into several parts.

\subsection{Application of Perron's formula} 
The first step is to apply Perron's formula twice to  
$M_Q(X, Y)$.  

\begin{lemma}\label{Perron_Formula:MXY}
For $X\ge 2$ and $Y\ge 2$ we have
\begin{equation}\label{def:MXY}
M_Q(X, Y)
= Y^{k-1}M(X, Y),
\end{equation}
where $(\kappa) := (\kappa-{\rm i}\infty, \kappa+{\rm i}\infty)$ and
$$
M(X, Y)
:= \frac{1}{(2\pi {\rm i})^2} 
 \int_{(\lambda)} 
\bigg(\int_{(\kappa)} \mathfrak{F}(s, z+k-1) 
\frac{X^{s+1} \d s}{s(s+1)}
\bigg)
\frac{Y^{z+1} \d z}{(z+k-1)(z+k)}\cdot
$$
\end{lemma}

\begin{proof}
Applying the Perron's formula \cite[Theorem II.2.3]{Tenenbaum1995} twice, we can get 
$$
M_Q(X, Y)
= \frac{1}{(2\pi {\rm i})^2} 
 \int_{(\lambda_k)} 
\bigg(\int_{(\kappa)} \mathfrak{F}(s, w) 
\frac{X^{s+1}}{s(s+1)} \d s
\bigg)
\frac{Y^{w+1}}{w(w+1)} \d w,
$$
which implies the required formula by changement of variable $w=z+k-1$.
\end{proof}

\subsection{Application of Cauchy's theorem} 
In this subsection, we shall apply Cauchy's theorem to evaluate the integral in $s$ in $M(X, Y)$.
Set
\begin{equation}\label{def:sjz} 
s_j(z) := \frac{1-(j-1)z}{4-j},
\quad
(j=1, 2, 3)
\end{equation} 
and
\begin{equation}\label{def:Fk*} 
\begin{cases}
\displaystyle
\mathfrak{F}_1^*(z)
:= \frac{1}{3}\cdot 
\frac{\zeta(3z) \zeta(2z+\frac{1}{3}) \zeta(z+\frac{2}{3}) \mathfrak{G}(s_1(z), z+k-1)}{s_1(z)(s_1(z)+1)},
\\\noalign{\vskip 1,2mm}
\displaystyle
\mathfrak{F}_2^*(z)
:= \frac{1}{2}\cdot 
\frac{\zeta(3z) \zeta(\frac{1+3z}{2}) \zeta(\frac{3-3z}{2}) \mathfrak{G}(s_2(z), z+k-1)}{s_2(z)(s_2(z)+1)},
\\\noalign{\vskip 1,2mm}
\displaystyle
\mathfrak{F}_3^*(z)
:= \frac{\zeta(3z) \zeta(2-3z) \zeta(3-6z) \mathfrak{G}(s_3(z), z+k-1)}{s_3(z)(s_3(z)+1)}\cdot
\end{cases}
\end{equation}

\begin{lemma}\label{Lem:MXYTU}
For $X\ge 2$, $Y\ge 2$ and $\sigma_0\in (\tfrac{1}{6}, \tfrac{1}{4})$ we have
\begin{equation}\label{Evaluate:MXY}
M(X, Y)
= I_1 + I_2 + I_3 + R_0(X, Y),
\end{equation}
where 
\begin{equation}\label{def:Ij}
I_j :=  \frac{1}{2\pi {\rm i}} \int_{(\lambda)} 
\frac{\mathfrak{F}_j^*(z) X^{(5-j-(j-1)z)/(4-j)} Y^{z+1}}{(z+k-1)(z+k)} \d z
\end{equation}
and
\begin{equation}\label{def:R0XY}
R_{ 0}(X, Y)
:= \frac{1}{(2\pi\mathrm{i})^2} \int_{(\lambda)} 
\bigg(\int_{(\sigma_0)} \frac{\mathfrak{F}(s, z+k-1) X^{s+1}}{s(s+1)} \d s\bigg)
\frac{Y^{z+1}}{(z+k-1)(z+k)} \d z.
\end{equation}
Further we have
\begin{equation}\label{UB_TR0}
\left.
\begin{array}{rl}
(\mathscr{D}R_0)(X, X+H; Y, Y+J)\!
\\\noalign{\vskip 1mm}
(\mathscr{D}R_0)(X-H, X; Y-J, Y)\!
\end{array}
\right\}
\ll_{k, \varepsilon, \sigma_0} X^{\sigma_0+\varepsilon} Y^{\frac{1}{3}+\varepsilon} H J
\end{equation}
uniformly for $(X, Y, H, J)$ in \eqref{Condition:XYHJ},
where the implied constant depends on $k$, $\varepsilon$ and $\sigma_0$.
\end{lemma}

\begin{proof}
We want to calculate the integral
$$
\frac{1}{2\pi {\rm i}} \int_{(\kappa)}
\frac{\mathfrak{F}(s, z+k-1) X^{s+1}}{s(s+1)} \d s
$$
for any individual $z=\lambda+\mathrm{i}v$ with $v\in \R$. And we shall cut the above infinite integral into finite one.
Let $T\ge (XY)^2 + |v|$.
The residue theorem allows us to deform the segment of the integration $[\kappa-\text{i}T, \kappa+\text{i}T]$ 
into the path joining the end-points 
$\kappa-\text{i}T$, $\sigma_0-\text{i}T$, $\sigma_0+\text{i}T$ and $\kappa+\text{i}T$.
The points $s_j(z)\;(j=1, 2, 3)$, given by \eqref{def:sjz},
are the simple poles of the integrand in the rectangle 
$\sigma_0\le \sigma\le \kappa$ and $|\tau|\le T$.
The residue of $\frac{\mathfrak{F}(s, z+k-1)}{s(s+1)} X^{s+1}$ at the poles $s_j(z)$ is equal to
\begin{equation}\label{def:residue}
\mathfrak{F}_j^*(z) X^{(5-j-(j-1)z)/(4-j)},
\end{equation}
where $\mathfrak{F}_j^*(z) \; (j=1, 2, 3)$ are defined as in \eqref{def:Fk*}.

In order to control the contribution of the segments $[\sigma_0\pm\text{i}T, \kappa\pm\text{i}T]$ 
and of the half-line $[\kappa\pm\text{i}T, \kappa\pm\text{i}\infty)$, we need the well-known estimate 
(cf. e.g. \cite[page 146, Theorem II.3.7]{Tenenbaum1995}) 
\begin{equation}\label{UB:zeta}
\zeta(s)\ll |\tau|^{\max\{(1-\sigma)/3, 0\}} \log |\tau|
\end{equation}
for $\frac{1}{2}\le \sigma<1, \; \tau\in \R$.

When $\sigma_0\le \sigma\le \kappa$ and $u=\lambda$,
it is easy to check that
\begin{equation}\label{Cond:sigmau}
\min_{0\le j\le 3} ((3-j)\sigma+ju)
\ge \tfrac{1}{2}+\varepsilon,
\end{equation}
which coincides with \eqref{cond:sigmau}.
It follows from \eqref{UB:zeta} and \eqref{UB:Gsw} that,  
for $\sigma_0\le \sigma\le \kappa$, $\tau\in \R$ and $u=\lambda, v\in \R$,
$$
\mathfrak{F}(s, z+k-1)
\ll_{k, \varepsilon, \sigma_0} (|\tau|+|v|+3)^{\frac{2}{3}-2\sigma+3\varepsilon}. 
$$
This implies
\begin{align*}
\int_{\sigma_0\pm{\rm i}T}^{\kappa\pm{\rm i}T} 
\frac{\mathfrak{F}(s, z+k-1) X^{s+1}}{s(s+1)} \d s 
& \ll_{k, \varepsilon, \sigma_0} \frac{X}{T^{\frac{4}{3}-\varepsilon}} 
\int_{\sigma_0}^{\kappa}\bigg(\frac{X}{T^2}\bigg)^{\sigma} \d \sigma
\ll_{k, \varepsilon, \sigma_0} \frac{X}{T},
\\
\int_{\kappa\pm{\rm i}T}^{\kappa\pm{\rm i}\infty} 
\frac{\mathfrak{F}(s, z+k-1) X^{s+1}}{s(s+1)} \d s
& \ll_{k, \varepsilon, \sigma_0} X^{4/3+\varepsilon} 
\int_{T}^{\infty} \frac{\d \tau}{\tau^{2-\varepsilon}}
\ll_{k, \varepsilon, \sigma_0} \frac{X^{4/3+\varepsilon}}{T^{1-\varepsilon}}\cdot
\end{align*}
Cauchy's theorem then gives 
\begin{align*}
\frac{1}{2\pi {\rm i}} \int_{(\kappa)} \frac{\mathfrak{F}(s, z+k-1) X^{s+1}}{s(s+1)} \d s
& = \sum_{j=1}^3 \mathfrak{F}_j^*(z) X^{(5-j-(j-1)z)/(4-j)}
\\
& + \frac{1}{2\pi {\rm i}} \int_{\sigma_0-{\rm i}T}^{\sigma_0+{\rm i}T} 
\frac{\mathfrak{F}(s, z+k-1) X^{s+1}}{s(s+1)} \d s
+ O_{k, \varepsilon, \sigma_0}\bigg(\frac{X^2}{\sqrt{T}}\bigg).
\end{align*}
Making $T\to\infty$, we find that
\begin{equation}\label{integral}
\begin{aligned}
\frac{1}{2\pi {\rm i}} \int_{(\kappa)} \frac{\mathfrak{F}(s, z+k-1) X^{s+1}}{s(s+1)} \d s
& = \sum_{j=1}^3 \mathfrak{F}_j^*(z) X^{(5-j-(j-1)z)/(4-j)}
\\
& + \frac{1}{2\pi {\rm i}} \int_{(\sigma_0)} \frac{\mathfrak{F}(s, z+k-1) X^{s+1}}{s(s+1)} \d s,
\end{aligned}
\end{equation}
which implies \eqref{Evaluate:MXY}.

Finally we prove \eqref{UB_TR0}.
For $s=\sigma_0+{\rm i}\tau$ with $\tau\in \R$ and $z=\lambda+{\rm i}v$ with $v\in \R$, 
in view of \eqref{Cond:sigmau}, we can apply \eqref{UB:Gsw} to get 
$$
\mathfrak{F}(s, z+k-1)
\ll_{k, \varepsilon, \sigma_0} 
\zeta(3\sigma_0+\text{i}3\tau) 
\zeta(2\sigma_0+\lambda+\text{i}(2\tau+v)) 
\zeta(\sigma_0+2\lambda+\text{i}(\tau+2v)) 
$$
and, by Lemma \ref{Lem3.3} with $\beta=1-\varepsilon$, 
\begin{align*}
r_{s, z}(X, H; Y, J)
& := \big((X+H)^{s+1}-X^{s+1}\big) \big((Y+J)^{z+1}-Y^{z+1}\big)
\\
& \ll X^{\sigma_0+\varepsilon} Y^{\frac{1}{3}+\varepsilon} H J (|\tau|+1)^{1-\varepsilon} (|v|+1)^{1-\varepsilon}. 
\end{align*}
With the help of the well-known bound
\begin{equation}\label{2,4meanvalue}
\int_0^t \big(|\zeta(\sigma+\text{i}\tau)|^2+|\zeta(\sigma+\text{i}\tau)|^4\big) \d \tau\ll t\log^4(t+3),
\quad
(\tfrac{1}{2}\le \sigma\le 2, \, t\ge 0)
\end{equation}
and the H\"older inequality, we can derive that
\begin{align*}
& \hskip -2mm
\int_{-\infty}^{\infty} \frac{|\zeta(3\sigma_0+\text{i}3\tau) 
\zeta(2\sigma_0+\lambda+\text{i}(2\tau+v)) 
\zeta(\sigma_0+2\lambda+\text{i}(\tau+2v))|}
{(|\tau|+1)^{1+\varepsilon}} 
\d \tau
\\
& \ll \bigg(
\int_{-\infty}^{\infty} \frac{|\zeta(3\sigma_0+\text{i}3\tau)|^4}{(|\tau|+1)^{1+\varepsilon}}\d \tau
\int_{-\infty}^{\infty} \frac{|\zeta(2\sigma_0+\lambda+\text{i}(2\tau+v))|^4}{(|\tau|+1)^{1+\varepsilon}}\d \tau
\bigg)^{1/4}
\\
& \hskip 1,4mm\times
\bigg(\int_{-\infty}^{\infty} \frac{|\zeta(\sigma_0+2\lambda+\text{i}(\tau+2v))|^2}
{(|\tau|+1)^{1+\varepsilon}}\d \tau\bigg)^{1/2}
\\\noalign{\vskip 2mm}
& \ll_{k, \varepsilon, \sigma_0} 1.
\end{align*}
The last inequality is obtained by integration by parts.
These estimates and Lemma \ref{lem3.2}(ii) imply 
\begin{align*}
(\mathscr{D}R_0)(X, X+H; Y, Y+J)
& = \int_{(\lambda)} 
\int_{(\sigma_0)}  
\frac{\mathfrak{F}(s, z+k-1)r_{s, z}(X, H; Y, J)}{s(s+1)(z+k-1)(z+k)} \frac{\d s \d z}{(2\pi\mathrm{i})^2}
\\\noalign{\vskip 1mm}
& \ll_{k, \varepsilon, \sigma_0} X^{\sigma_0+\varepsilon} Y^{\frac{1}{3}+\varepsilon} H J.
\end{align*}
This completes the proof.
\end{proof}

\subsection{Evaluation of $I_1$ and $I_2$}

\begin{lemma}\label{Lem:Evaluate_I1}
For $X\ge 2$, $Y\ge 2$ and $u_0\in (\frac{1}{6}, \frac{1}{4})$ we have
\begin{equation}\label{Evaluate:I1}
I_1 = (XY)^{\frac{4}{3}} P_1( \log Y) + R_1(X, Y),
\end{equation}
where $P_1(t)$ is a quadratic polynomial and
\begin{equation}\label{def:R1XY}
R_1(X, Y)
:= \frac{1}{2\pi {\rm i}} \int_{(u_0)} 
\frac{\mathfrak{F}_1^*(z) X^{4/3} Y^{z+1}}{(z+k-1)(z+k)} \d z.
\end{equation}
Further we have
\begin{equation}\label{UB_TR1}
\left.
\begin{array}{rl}
(\mathscr{D}R_1)(X, X+H; Y, Y+J)\!
\\\noalign{\vskip 1mm}
(\mathscr{D}R_1)(X-H, X; Y-J, Y)\!
\end{array}\right\}
\ll_{k, \varepsilon, u_0} X^{\frac{1}{3}+\varepsilon} Y^{u_0+\varepsilon} H J
\end{equation}
uniformly for $(X, Y, H, J)$ in \eqref{Condition:XYHJ},
where the implied constant depends on $k$, $\varepsilon$ and $u_0$. 
\end{lemma}

\begin{proof}
We move the line of integration from $(\lambda)$ to $(u_0)$.
Obviously $z=\tfrac{1}{3}$  
is the unique pole of order 3 of the integrand in the strip $u_0\le u\le \lambda$, and 
the residue is $(XY)^{\frac{4}{3}} P_1(\log Y)$ with 
\begin{equation}\label{def:P1t}
P_1(t)
:= \frac{1}{2!} \bigg(
\frac{(z-\frac{1}{3})^3\mathfrak{F}_1^*(z) \mathrm{e}^{t(z-\frac{1}{3})}}{(z+k-1)(z+k)}
\bigg)''\bigg|_{z=\frac{1}{3}}.
\end{equation}
When $u_0\le u\le \lambda$, it is easy to check that
$$
\min_{0\le j\le 3} ((3-j)s_1(u)+ju)
\ge \tfrac{1}{2}+\varepsilon,
\qquad
\mathfrak{F}_1^*(z)\ll_{k, \varepsilon, u_0} (|v|+1)^{\frac{2}{3}-2u+\varepsilon}.
$$
Similar to \eqref{integral}, we can obtain \eqref{Evaluate:I1}.

To establish \eqref{UB_TR1}, we note that for $u=u_0$ we have, as before,
$$
\mathfrak{F}_1^*(z)
\ll_{k, \varepsilon, u_0} \zeta(3u_0+3\text{i}v) \zeta(2u_0+\tfrac{1}{3}+2\text{i}v) \zeta(u_0+\tfrac{2}{3}+\text{i}v)
$$
and, by Lemma \ref{Lem3.3} with $\beta=1-\varepsilon$, 
\begin{align*}
r_{s_1(z), z}(X, H; Y, J)
& := \big((X+H)^{4/3}-X^{4/3}\big) \big((Y+J)^{z+1}-Y^{z+1}\big)
\\
& \ll X^{\frac{1}{3}+\varepsilon} Y^{u_0+\varepsilon} H J (|v|+1)^{1-\varepsilon}. 
\end{align*}
With the help of \eqref{2,4meanvalue}, we can derive, as before, that
$$
\int_{-\infty}^{\infty} 
\frac{|\zeta(3u_0+3\text{i}v) \zeta(2u_0+\tfrac{1}{3}+2\text{i}v) \zeta(u_0+\tfrac{2}{3}+\text{i}v)|}{(|v|+1)^{1+\varepsilon}} \d v
\ll_{k, \varepsilon, u_0} 1.
$$
Combining these with Lemma \ref{lem3.2}(ii), we deduce that 
\begin{align*}
(\mathscr{D}R_1)(X, X+H; Y, Y+J)
& = \frac{1}{2\pi\text{i}} \int_{(u_0)} 
\frac{\mathfrak{F}_1^*(z)r_{s_1(z), z}(X, H; Y, J)}{(z+k-1)(z+k)} \d z
\\
& \ll_{k, \varepsilon, u_0} X^{\frac{1}{3}+\varepsilon} Y^{u_0+\varepsilon} H J, 
\end{align*}
from which the desired result follows. 
\end{proof}

\vskip 1mm

\begin{lemma}\label{Lem:Evaluate_I2}
For $X\ge 2$, $Y\ge 2$ and $u_0\in (\tfrac{1}{6}, \tfrac{1}{4})$ we have
\begin{equation}\label{Evaluate:I2}
I_2 = (XY)^{\frac{4}{3}} P_2(\log(X^{-\frac{1}{2}}Y)) + R_2(X, Y), 
\end{equation}
where $P_2(t)$ is quadratic polynomial and
\begin{equation}\label{def:R2XY}
R_2(X, Y)
:= \frac{1}{2\pi {\rm i}} \int_{(u_0)} 
\frac{\mathfrak{F}_2^*(z) X^{\frac{1}{2}(3-z)} Y^{z+1}}{(z+k-1)(z+k)} \d z.
\end{equation}
Further we have
\begin{equation}\label{UB_TR2}
\left.
\begin{array}{rl}
(\mathscr{D}R_2)(X, X+H; Y, Y+J)\!
\\\noalign{\vskip 1mm}
(\mathscr{D}R_2)(X-H, X; Y-J, Y)\!
\end{array}\right\}
\ll_{k, \varepsilon, u_0} X^{\frac{1}{2}(1-u_0)} Y^{u_0} HJ
\end{equation}
uniformly for $(X, Y, H, J)$ in \eqref{Condition:XYHJ},
where the implied constant depends on $k$, $\varepsilon$ and $u_0$.
\end{lemma}

\begin{proof}
The proof is rather simlar to that of Lemma \ref{Lem:Evaluate_I1} and even simpler.
We move the line of integration from $(\lambda)$ to $(u_0)$.
Obviously $z=\tfrac{1}{3}$ is the unique pole of order 3 of the integrand in the strip $u_0\le u\le \lambda$, and 
the residue is $(XY)^{\frac{4}{3}} P_2\big(\log(X^{-\frac{1}{2}}Y)\big)$ with 
\begin{equation}\label{def:P2t}
P_2(t)
:= \frac{1}{2!} \bigg(
\frac{(z-\frac{1}{3})^3\mathfrak{F}_2^*(z) {\rm e}^{t(z-\frac{1}{3})}}{(z+k-1)(z+k)}
\bigg)''\bigg|_{z=\frac{1}{3}}.
\end{equation}

When $u_0\le u\le \lambda$, it is easy to check that
$$
\min_{0\le j\le 3} ((3-j)s_2(u)+ju)
\ge \tfrac{1}{2}+\varepsilon.
$$
It follows from \eqref{UB:zeta} and \eqref{UB:Gsw} that, for $u_0\le u\le \lambda$, 
$$
\mathfrak{F}_2^*(z)\ll_{k, \varepsilon, u_0} (|v|+1)^{\frac{1}{3}-u-2+\varepsilon}.
$$
These imply \eqref{Evaluate:I2}.
Further by Lemma \ref{Lem3.3} with $\beta=1$, 
\begin{align*}
r_{s_2(z), z}(X, H; Y, J)
& := \big((X+H)^{\frac{1}{2}(3-z)}-X^{\frac{1}{2}(3-z)}\big)\big((Y+J)^{z+1}-Y^{z+1}\big)
\\
& \ll X^{\frac{1}{2}(1-u_0)} Y^{u_0} HJ (|v|+1)^2. 
\end{align*}
Combining these with  Lemma \ref{lem3.2}(ii), we deduce that
\begin{align*}
(\mathscr{D}R_2)(X, X+H; Y, Y+J)
& = \frac{1}{2\pi\text{i}} \int_{(u_0)}  
\frac{\mathfrak{F}_2^*(z) r_{s_2(z), z}(X, H; Y, J)}{(z+k-1)(z+k)} \d z
\\
& \ll_{k, \varepsilon, u_0} X^{\frac{1}{2}(1-u_0)} Y^{u_0} HJ.
\end{align*}
This completes the proof.
\end{proof}

\subsection{Evaluation of $I_3$}

\begin{lemma}\label{Lem:Evaluate_I3}
For $X\ge 2$ and $Y\ge 2$ we have
\begin{equation}\label{Evaluate:I3}
I_3 = (XY)^{\frac{4}{3}} P_3\big(\log(X^{-2}Y)\big) + R_3(X, Y)
\end{equation}
where $P_3(t)$ is defined as in \eqref{def:P3t} below and
\begin{equation}\label{def:R3XY}
R_3(X, Y)
:= \frac{(XY)^{\frac{4}{3}}}{2\pi {\rm i}} \int_{(\lambda)} 
\frac{\mathfrak{F}_3^*(z) \big({\rm e}^{\xi (z-\frac{1}{3})}-\sum_{j=0}^2 \frac{1}{j!}\xi^j(z-\frac{1}{3})^j\big)}
{(z+k-1)(z+k)} \d z
\end{equation}
with $\xi=\log(X^{-2}Y)$.
Further we have
\begin{equation}\label{UB_TR3}
\left.
\begin{array}{rl}
(\mathscr{D}R_3)(X, X+H; Y, Y+J)\!
\\\noalign{\vskip 1mm}
(\mathscr{D}R_3)(X-H, X; Y-J, Y)\!
\end{array}\right\}
\ll_k \big(X^{-\frac{5}{3}} Y^{\frac{4}{3}}H^3  + X^{\frac{4}{3}} Y^{-\frac{5}{3}}J^3 \big) \mathcal{L}^4
\end{equation}
uniformly for $(X, Y, H, J)$ in \eqref{Condition:XYHJ},
where the implied constant depends on $k$ at most.
\end{lemma}

\begin{proof}
Putting $\xi=\log(X^{-2}Y)$, we can write
$$
I_3 
= \frac{(XY)^{\frac{4}{3}}}{2\pi {\rm i}} \int_{(\lambda)} 
\frac{\mathfrak{F}_3^*(z) {\rm e}^{\xi (z-\frac{1}{3})}}{(z+k-1)(z+k)} \d z
= (XY)^{\frac{4}{3}} P_3(\xi) + R_3(X, Y),
$$
where 
\begin{equation}\label{def:P3t}
P_3(t) = \sum_{j=0}^2 \frac{a_j}{j!} t^j
\quad\text{with}\quad
a_j := \frac{1}{2\pi {\rm i}} \int_{(\lambda)} 
\frac{\mathfrak{F}_3^*(z) (z-\frac{1}{3})^j}{(z+k-1)(z+k)} \d z.
\end{equation}
On the other hand, for $\Re e\, z =\lambda$, we have 
$|\text{e}^{\xi (z-\frac{1}{3})}|=\text{e}^{\xi (\lambda-\frac{1}{3})}\asymp 1$.
Thus we can write 
$$
{\rm e}^{\xi (z-\frac{1}{3})} - \sum_{j=0}^2 \frac{1}{j!} \xi^j (z-\tfrac{1}{3})^j
= O(\xi^3 |z-\tfrac{1}{3}|^3).
$$
From this and the bound $\mathfrak{F}_3^*(z)\ll \mathcal{L}^3(1+|v|)^{-2} \; (\Re e\, z =\lambda)$,
it is easy to deduce that
\begin{align*}
R_3(X, Y)
& \ll (XY)^{\frac{4}{3}} |\xi|^3 \mathcal{L}^4.
\end{align*}
According to Lemma \ref{lem3.2} {\rm (ii)}, $(\mathscr{D}R_3)(X, X+H; Y, Y+J)$ is actually a finite linear combination of values of $R_3(X, Y)$ for $\xi\ll H/X+J/Y$.
Hence we have 
\begin{align*}
(\mathscr{D}R_3)(X, X+H; Y, Y+J)\!
& \ll_k  (XY)^{\frac{4}{3}}((H/X)^3+(J/Y)^3)  \mathcal{L}^4.
\end{align*}
This proves the lemma. 
\end{proof} 

\subsection{Completion of proof of Proposition \ref{Asymp:SQxySWxy}} 
Denote by $\mathcal{M}_Q(X, Y)$ the main term in the asymptotic formula 
of $M_Q(X, Y)$ in Proposition~\ref{Pro:MQXY}, that is 
$$
\mathcal{M}_Q(X, Y) := X^{\frac{4}{3}} Y^{k+\frac{1}{3}} 
\big(P_1(\log Y) + P_2(\log(X^{-\frac{1}{2}}Y)) + P_3(\log(X^{-2}Y))\big).
$$ 
Then Lemma \ref{lem3.2}(i) gives 
\begin{equation}\label{def:PQ(t,u)}
\begin{aligned}
(\mathscr{D}\mathcal{M}_Q)(X, X+H; Y, Y+J)
& = Y^{k-1} \big\{(XY)^{\frac{1}{3}} P_Q(\log X, \log Y)
\\
& \quad
+ O(X^{\frac{1}{3}}Y^{-\frac{2}{3}}J\mathcal{L}^3+X^{-\frac{2}{3}}Y^{\frac{1}{3}}H\mathcal{L}^3)\big\}HJ.
\end{aligned}
\end{equation}
Since $\mathscr{D}$ is a linear operator, 
this together with Proposition~\ref{Pro:MQXY}, 
\eqref{UB_TR0}, \eqref{UB_TR1}, \eqref{UB_TR2} and \eqref{UB_TR3} 
with the choice of $\sigma_0=u_0=\tfrac{1}{6}+\varepsilon$
implies that
\begin{equation}\label{DMXHYJ}
(\mathscr{D}M_Q)(X, X+H; Y, Y+J)
= Y^{k-1} \big\{(XY)^{\frac{1}{3}} P_Q(\log X, \log Y) + O_{\varepsilon}(\mathcal{R} (XY)^{\varepsilon})\big\}HJ
\end{equation}
with
\begin{align*}
\mathcal{R}
& := X^{\frac{1}{6}} Y^{\frac{1}{3}}
+ X^{\frac{5}{12}} Y^{\frac{1}{6}} 
+ X^{-\frac{5}{3}} Y^{\frac{4}{3}} H^2 J^{-1} 
+ X^{\frac{4}{3}} Y^{-\frac{5}{3}}H^{-1} J^2  
+ X^{\frac{1}{3}}Y^{-\frac{2}{3}}J
+ X^{-\frac{2}{3}}Y^{\frac{1}{3}}H.
\end{align*}
And the same formula also holds for $(\mathscr{D}M_Q)(X-H, X; Y-J, Y)$.

Taking $H=X^{\frac{5}{6}}$ and $J=X^{-\frac{1}{6}}Y$,
Lemma \ref{lem3.1} and \eqref{DMXHYJ} give us
\begin{align*}
S_Q(X, Y) 
= Y^{k-1} \big\{(XY)^{\frac{1}{3}} P_Q(\log X, \log Y) 
+ O_{\varepsilon}(X^{\frac{1}{6}+\varepsilon}Y^{\frac{1}{3}+\varepsilon} + X^{\frac{5}{12}+\varepsilon} Y^{\frac{1}{6}+\varepsilon})\big\}.
\end{align*}
The estimate for $S_W(x,y)$ can be proved in the same way, and the only difference is the leading coefficients of polynomials.
This completes the proof of Proposition \ref{Asymp:SQxySWxy}.

\vskip 8mm

\section{Proofs of Theorems \ref{thm1} and \ref{thm2}}

By \eqref{def:NdaggerQ(B)} and \eqref{def:SQxy}, it follows that
$$
N^*_Q(B)
= \frac{2(2\pi)^{\frac{m}{2}}}{\Gamma(\frac{m}{2})\sqrt{|A|}} S_Q(B, B^2)
+ O_{\varepsilon}\big(m^{\frac{3m}{4}}\|Q\|^{\frac{m}{4}}B^{\frac{m+1}{2}+\varepsilon}\big)
$$
where we have used the following bound
\begin{equation}\label{UB:error}
\sum_{1\leq b \leq B} \sum_{1\leq n \leq B^2} \mathbb{1}_3(bn) n^{\frac{m-1}{4}+\varepsilon}
\ll\sum_{h\le B^3} \mathbb{1}_3(h) \sum_{n| h, \, n\le B^2} n^{\frac{m-1}{4}+\varepsilon}
\ll B^{\frac{m+1}{2}+\varepsilon}
\end{equation}
and the implied constant depends on $\varepsilon$ only.
The first formula in \eqref{NQB} follows immediately 
from \eqref{Cor:SQxx2} of Proposition \ref{Asymp:SQxySWxy}. 
We deduce the second formula in \eqref{NQB} from this and the inversion formula \eqref{MobInv} 
with $P_Q(t)$ determined by the following relation
\begin{equation}\label{def:PQ(t)}
\sum_{d\ge 1} \frac{\mu(d)}{d^{m-1}} P^*_Q\bigg(\log\bigg(\frac{B^{1/(m-1)}}{d}\bigg)\bigg)
= \frac{1}{(m-1)^2\zeta(m-1)} P_Q(\log B).
\end{equation}
We note that $P^*_Q(t)$ and $P_Q(t)$ have the same leading coefficients.
This proves Theorem \ref{thm1}. 

\medskip 

From \eqref{def:NdaggerQ(B)}, \eqref{formula:r(n,Q)} of Proposition \ref{Prop/Q=n},
\eqref{SinSer=:2}, Proposition \ref{SinSerEv} and \eqref{UB:error}, 
we can deduce that
$$
N^*_Q(B)
\le \varpi^+ \mathcal{C}^*_Q S_W(B, B^2)
+ O_{\varepsilon}\big(m^{\frac{3m}{4}}\|Q\|^{\frac{m}{4}}B^{\frac{m+1}{2}+\varepsilon}\big). 
$$
By \eqref{Cor:SWxx2} of Proposition \ref{Asymp:SQxySWxy}, we obtain that
$$
N^*_Q(B)
\le \varpi^+ \mathcal{C}^*_Q W^*.
$$
Similarly we can prove that $N^*_Q(B)\ge \varpi^- \mathcal{C}^*_Q W^*$.
This proves Theorem \ref{thm2}.

\vskip 8mm

\section{Proof of Corollary \ref{cor:sum-squares}}

In order to prove \eqref{cor:N*mB-NmB},
it is sufficient to show that when $Q=y_1^2+\cdots+y_m^2$ with $m=4k$, we have
\begin{equation}\label{cor:sum-squares:constant}
\mathcal{C}^*_Q \mathscr{C}_Q = \mathcal{C}_m^*.
\end{equation} 
For this, firstly let us recall some notations of \cite[pages 2039--2040]{LWZ19}:
\begin{equation}\label{G2}
\begin{aligned}
\mathscr{G}_2(s, w)
& := \prod_{1\le j\le 3} (1-2^{-(s+jw-j(2k-1))})
\\
& \quad
\times \bigg(1
+ a\frac{1+2^{-w+2k-1}+2^{-2w+2(2k-1)}}{2^{s+w-(2k-1)}-2^{-2w+2(2k-1)}} 
- b\frac{2^{-s-w}(1+2^{-w}+2^{-2w})}{1-2^{-s-3w}}\bigg),
\end{aligned}
\end{equation}
and
\begin{equation}\label{Gp}
\begin{aligned}
\mathscr{G}_p(s, w)
& := \bigg(
1
+ \frac{p^{2k-1}+1}{p^{s+w}}
+ \frac{p^{2(2k-1)}+p^{2k-1}+1}{p^{s+2w}}
+ \frac{p^{4k-2}+p^{2k-1}}{p^{s+3w}}
+ \frac{p^{4k-2}}{p^{2s+4w}}
\bigg)
\\
& \qquad
\times
\bigg(1-\frac{p^{2k-1}}{p^{s+w}}\bigg)
\bigg(1-\frac{p^{2(2k-1)}}{p^{s+2w}}\bigg)
\bigg(1-\frac{1}{p^{s+3w}}\bigg)^{-1}
\end{aligned}
\end{equation}
with
\begin{equation}\label{ab}
a := 1-\dfrac{(-1)^{k}}{1-2^{2k-1}}, 
\qquad
b := (-1)^{k} \dfrac{1-2^{2k}}{1-2^{2k-1}}\cdot
\end{equation}

We compute the 2-part of $\mathscr{C}_Q$ first.
By \eqref{sum-square:delta2}, we have 
\begin{align*}
\delta_2(2^\nu,Q)
& = 1 + (-1)^k \Big(\sum_{2\leq r \leq \nu}2^{(2k-1)(1-r)}-2^{-(2k-1)\nu}\Big)
= \begin{cases}
1 & \text{if $\nu=0$},
\\
a - b\cdot 2^{(1-2k)\nu} &  \text{if $\nu\geq 1$}.
\end{cases}
\end{align*}
Then we get 
$$
\sum_{\nu=0}^{3d} \delta_2(2^{\nu}, Q)
= 1 + 3ad
+ \frac{b}{1-2^{2k-1}}\big(1-2^{(1-2k)3d}\big)
$$
and
\begin{align*}
\sum_{d=0}^{\infty} \frac{\sum_{\nu=0}^{3d} \delta_2(2^{\nu}, Q)}{2^d}
& = 2 + 6a
+ \frac{b}{1-2^{2k-1}} \bigg(1-\frac{2^{2-6k}}{1-2^{2-6k}}\bigg)
\\
& = 2 \bigg(1 + 3a
- b\frac{2^{-2k}(1+2^{1-2k}+2^{2(1-2k)})}{1-2^{2-6k}}\bigg).
\end{align*}
So we have 
\begin{equation}\label{proof:cor-1}
\begin{aligned}
\bigg(1-\frac{1}{2}\bigg)^4 \sum_{d\geq 0} \frac{\sum_{\nu=0}^{3d} \delta_2(2^{\nu}, Q)}{2^d}
& = \mathscr{G}_2(1,2k-1).
\end{aligned}
\end{equation}
Since $D=2^m$, we have $\chi(p)=\big(\frac{4D}{p}\big)=\big(\frac{2^{\frac{m}{2}+1}}{p}\big)^2=1$ for all odd primes $p$. Thus 
\begin{equation}\label{proof:cor-2}
\bigg(1+\frac{2}{p}+\frac{3\chi}{p^\frac{m}{2}}+\frac{2\chi^2}{p^{m-1}}+\frac{\chi^2}{p^m}\bigg)
\bigg(1-\frac{1}{p}\bigg)^2 \bigg(1-\frac{1}{p^{\frac{3m}{2}-2}}\bigg)^{-1}
= \mathscr{G}_p(1,2k-1),
\end{equation}
since 
$$
L\big(\tfrac{3m}{2}-2,\chi\big)=\prod_{p\geq3}\bigg(1-\frac{1}{p^{\frac{3m}{2}-2}}\bigg)^{-1}.
$$
Noticing that $|A|=2^m$, formulae \eqref{Constants:calC*QCQ}, \eqref{def:scrCQ}, \eqref{proof:cor-1} and \eqref{proof:cor-2} allow us to write
$$
\mathcal{C}^*_Q \mathscr{C}_Q 
= \frac{\pi^{\frac{m}{2}}(1-2^{-\frac{m}{2}})^{-1}}{\Gamma(\frac{m}{2})(3m-4)\zeta(\frac{m}{2})}
\prod_{p} \mathscr{G}_p(1,2k-1).
$$
This implies the required formula \eqref{cor:sum-squares:constant} 
thanks to the relation $\zeta(m)=\frac{|B_m|(2\pi)^m}{2\cdot m!}$.

\vskip 8mm

\section{Proof of Corollary \ref{cor:level-one}}

In order to prove \eqref{cor:N*EB-NEB},
it is sufficient to show that when $Q$ is a quadratic form of level one in $m\equiv 0\,({\rm mod}\,8)$ variables, we have
\begin{equation}\label{cor:level-one:constant}
\mathcal{C}^*_Q \mathscr{C}_Q = \mathcal{C}^*_E.
\end{equation} 

For this, we compute the 2-part of $\mathscr{C}_Q$ first.
In view of \eqref{level-one:delta2}, we have 
\begin{equation}\label{proof:cor2-1}
\begin{aligned}
\sum_{d\geq 0}\frac{\sum_{\nu=0}^{3d}\delta_2(2^{\nu}, Q)}{2^d}
& = \frac{1-2^{-\frac{m}{2}}}{1-2^{1-\frac{m}{2}}}
\sum_{d\geq0} \frac{1}{2^d} 
\bigg(3d+1-\frac{2^{1-\frac{m}{2}}}{1-2^{1-\frac{m}{2}}} \big(1-2^{(1-\frac{m}{2})(3d+1)}\big)\bigg)
\\
& = \frac{1-2^{-\frac{m}{2}}}{1-2^{1-\frac{m}{2}}}
\bigg(8-\frac{2^{1-\frac{m}{2}}}{1-2^{1-\frac{m}{2}}}
\bigg(2-\frac{2^{1-\frac{m}{2}}}{1-2^{2-\frac{3m}{2}}}\bigg)\bigg)
\\
& = \frac{4(1-2^{-\frac{m}{2}})}{1-2^{-(\frac{3m}{2}-2)}}
\bigg(1 + \frac{2}{2} + \frac{3}{2^\frac{m}{2}} + \frac{2}{2^{m-1}} + \frac{1}{2^m}\bigg).
\end{aligned}
\end{equation}

Since $Q$ is of level one, 
we have $|A|=|D|=1$ and $\chi(p)=\big(\frac{4}{p}\big)=\big(\frac{2}{p}\big)^2=1$ for all odd primes $p$.
Thus 
\begin{equation}\label{proof:cor2-2}
\bigg(1-\frac{1}{p}\bigg)^2
\bigg(1+\frac{2}{p}+\frac{3\chi}{p^\frac{m}{2}}+\frac{2\chi^2}{p^{m-1}}+\frac{\chi^2}{p^m}\bigg)
= \bigg(1-\frac{1}{p}\bigg)^2
\bigg(1+\frac{2}{p}+\frac{3}{p^\frac{m}{2}}+\frac{2}{p^{m-1}}+\frac{1}{p^m}\bigg)
\end{equation}
and
\begin{equation}\label{proof:cor2-3}
\frac{L(\frac{3}{2}m-2,\chi^3)}{L(\frac{m}{2}, \chi)}\cdot
\frac{1-2^{-\frac{m}{2}}}{1-2^{-(\frac{3m}{2}-2)}}
= \frac{\zeta(\frac{3}{2}m-2)}{\zeta(\frac{m}{2}, \chi)}\cdot
\end{equation}
Inserting \eqref{proof:cor2-1}, \eqref{proof:cor2-2} and \eqref{proof:cor2-3} into \eqref{def:scrCQ}, we obtain
$$
\mathscr{C}_Q
= \frac{\zeta(\frac{3}{2}m-2)}{(6m-8)\zeta(\frac{m}{2})}
\prod_{p} \bigg(1-\frac{1}{p}\bigg)^2
\bigg(1+\frac{2}{p}+\frac{3}{p^\frac{m}{2}}+\frac{2}{p^{m-1}}+\frac{1}{p^m}\bigg).
$$
This implies the required formula \eqref{cor:level-one:constant} 
since $\mathcal{C}^*_Q = 2(2\pi)^\frac{m}{2}/\Gamma(\frac{m}{2})$ in this case.

\medskip 

\noindent{\bf Acknowledgements.} The authors would like to thank Yongqiang Zhao for deep discussions. 
While writing this paper, the authors paid multiple visits to each 
other in the past years, despite the difficulty caused by the Coronavirus. An important part of this paper 
was done during a visit of the second auther to Universit\'e Paris-Est Cr\'eteil. The manuscript was completed 
in a visit of one of the authors to a beautiful campus besides the Summer Palace. 
It is a pleasure to record our gratitude to all these institutions 
for their hospitality and support. 
This work is supported by the National Natural Science Foundation of China 
(Grant Nos. 12031008, 12071375, 11871307, 11971370 and 11771252), and the second author is supported by 
the China Scholarship Council.

\end{document}